\documentclass[a4paper]{amsart}

\usepackage[hyphens]{url} 
\usepackage[dvips]{graphicx} 
\usepackage[new]{old-arrows}

\usepackage{amsthm, amsmath, amssymb, amsfonts} 
\usepackage{mathtools, mathabx, epsfig, mathrsfs}
\usepackage{physics}
\usepackage{import}

\usepackage{xcolor}
\usepackage{xspace}
\usepackage{stackrel}
\usepackage{makecell}
\usepackage{enumerate}
\usepackage{footnote}
\usepackage[textsize=tiny]{todonotes}

\usepackage{hyperref}
\hypersetup{
    colorlinks,
    linkcolor={red!70!black},
    citecolor={blue!50!black},
    urlcolor={blue!80!black}
}
\usepackage{url}

\usepackage{tikz,tqft}
\usetikzlibrary{matrix,arrows,cd,decorations.markings,calc,patterns,angles,quotes,backgrounds}

\theoremstyle{plain}
 \newtheorem{thm}{Theorem}[section]
 \newtheorem{thmx}{Theorem}
 
 \newtheorem{prop}[thm]{Proposition}
 \newtheorem{lem}[thm]{Lemma}
 \newtheorem{cor}[thm]{Corollary}
\theoremstyle{definition}
 
 \newtheorem{dfn}[thm]{Definition}
 \newtheorem{conv}[thm]{Convention}
\theoremstyle{remark}
 \newtheorem{rem}[thm]{Remark}
 \numberwithin{equation}{section}
 
\def\acts{\mathrel{\reflectbox{$\righttoleftarrow$}}}

\newcommand{\Q}{\mathbb{Q}}
\newcommand{\R}{\mathbb{R}}
\newcommand{\C}{\mathbb{C}} 	
\newcommand{\CP}{\mathbb{CP}}
\newcommand{\Z}{\mathbb{Z}}
  
\newcommand{\Hyp}{\mathbb{H}}

\renewcommand{\restriction}{\mathord{\upharpoonright}}

\newcommand{\deriv}{\left.\frac{d}{dt}\right\vert_{t=0}}

\newcommand\phantomarrow[2]{%
  \setbox0=\hbox{$\displaystyle #1\to$}%
  \hbox to \wd0{%
    $#2\mapstochar
     \cleaders\hbox{$\mkern-1mu\relbar\mkern-3mu$}\hfill
     \mkern-7mu\rightarrow$}%
  \,}
  
\DeclareMathOperator{\rot}{rot}
\DeclareMathOperator{\vol}{vol}

\DeclareMathOperator{\Hom}{Hom}
\DeclareMathOperator{\Rep}{Rep}

\DeclareMathOperator{\Ad}{Ad}

\DeclareMathOperator{\Aut}{Aut}

\DeclareMathOperator{\fix}{fix}

\DeclareMathOperator{\psl}{PSL(2,\R)}
\DeclareMathOperator{\SU2}{SU(2)}
\DeclareMathOperator{\PMCG}{Mod}
\newcommand{\E}{\mathcal{E}}
\newcommand{\surface}{\Sigma_{g,n}}
\newcommand{\smreplong}{\textnormal{Rep}^{\textnormal{\tiny{DT}}}_\alpha(\Sigma_n,\textnormal{PSL}(2,\R))\xspace}
\newcommand{\smrep}{\textnormal{Rep}^{\textnormal{\tiny{DT}}}_\alpha\xspace}
\newcommand{\co}{\mathfrak{c}}
\newcommand{\s}{\mathfrak{s}}

\renewcommand{\setminus}{\smallsetminus}

\title[Ergodicity and Deroin-Tholozan representations]{Ergodicity of the mapping class group action on Deroin-Tholozan representations}

\author[Arnaud Maret]{Arnaud Maret} 

\address{
Mathematisches Institut \\ 
Ruprecht-Karls-Universit\"at Heidelberg   \\ 
Germany}
\email{amaret@mathi.uni-heidelberg.de}

\begin{document}

\begin{abstract}
This note investigates the dynamics of the mapping class group action on compact connected components of relative character varieties of surface group representations into $\psl$, discovered by Deroin and Tholozan. We apply symplectic methods developed by Goldman and Xia to prove that the action is ergodic.
\end{abstract}

\maketitle

\section{Introduction}  

A character variety consists of conjugacy classes of representations of the fundamental group of a surface\footnote{A \textit{surface} is a two-dimensional manifold. A surface may have punctures. Unless otherwise stated, surfaces are orientable and have negative Euler characteristic.} $\Sigma$ into a Lie group $G$. Character varieties, or part of them, enjoy a natural symplectic structure provided that $G$ is, for instance, compact or semisimple. In this case, the mapping class group of $\Sigma$ acts by symplectomorphisms on the character variety (see Subsection \ref{sec:character_varieties}). The action is known to be ergodic if $G$ is compact, whereas its dynamical nature remains widely unknown if $G$ is not compact, as for instance $G=\psl$. It is nevertheless proven to be proper and discontinuous on the Teichm\"uller components of the $\psl$-character variety of a closed surface and conjectured to be ergodic on the remaining components (see Subsection \ref{sec:history}). This paper investigates the mapping class group action on some particularly nice components of the $\psl$-character variety of a punctured sphere.

\subsection{The result}

Let $n\geq 3$ be an integer and $\alpha=(\alpha_1,\ldots,\alpha_n)\in(0,2\pi)^n$ be an $n$-tuple of real numbers. We fix a collection $\{c_1,\ldots,c_n\}$ of generators of the fundamental group $\pi_1(\Sigma_{0,n})$ of the $n$-punctured sphere $\Sigma_{0,n}$ that satisfy the sole relation $c_1\cdot\ldots\cdot c_n=1$. Let $\Rep_\alpha(\Sigma_{0,n},\psl)$ denote the relative character variety of conjugacy classes of representations $\phi\colon \pi_1(\Sigma_{0,n})\to \psl$ for which $\phi(c_i)$ is an elliptic rotation of angle $\alpha_i$ for every $i$. Deroin-Tholozan proved in 
\cite{DeTh16} the existence of a nonempty compact connected component of $\Rep_\alpha(\Sigma_{0,n},\psl)$ whenever $\alpha_1+\ldots+\alpha_n>2\pi(n-1)$. We refer to these components as \textit{character variety of Deroin-Tholozan representations}\footnote{Deroin and Tholozan originally called these representations \emph{supra-maximal}. We prefer to call them \emph{Deroin-Tholozan} representations for the reasons explained in Remark \ref{rem:deroin-tholozan}.} denote them by
\[
\Rep^{\textnormal{\tiny{DT}}}_\alpha(\Sigma_{0,n},\psl)\subset \Rep_\alpha(\Sigma_{0,n},\psl).
\]
Let $\PMCG(\Sigma_{0,n})$ denote the group of isotopy classes of orientation-preserving homeomorphisms $\Sigma_{0,n}\to \Sigma_{0,n}$ that fix each puncture individually. Our main result is

\begin{thmx}\label{thm:ergodicity}
The action of $\PMCG(\Sigma_{0,n})$ on $\Rep^{\textnormal{\tiny{DT}}}_\alpha(\Sigma_{0,n},\psl)$ is ergodic with respect to the Goldman symplectic measure.
\end{thmx}

We prove Theorem \ref{thm:ergodicity} by applying methods developed in 
\cite{GoXi09} and in 
\cite{MaWo15}. The argument has a strong symplectic geometry flavour. The cornerstone of the proof relates the action of a Dehn twist in $\PMCG(\Sigma_{0,n})$ to a certain Hamiltonian flow on $\Rep^{\textnormal{\tiny{DT}}}_\alpha(\Sigma_{0,n},\psl)$, see Proposition \ref{prop:Dehn_twists_Ham_flow} for a precise statement. A coarse sketch of the proof consists of the following steps.

\begin{enumerate}
\item Identify sufficiently many Dehn twists in $\PMCG(\Sigma_{0,n})$ such that the associated Hamiltonian flows locally act transitively on $\Rep^{\textnormal{\tiny{DT}}}_\alpha(\Sigma_{0,n},\psl)$.
\item Prove that this implies that any integrable $\PMCG(\Sigma_{0,n})$-invariant function $\Rep^{\textnormal{\tiny{DT}}}_\alpha(\Sigma_{0,n},\psl)\to \R$ must be constant almost everywhere.
\end{enumerate}
Theorem \ref{thm:ergodicity} can be refined to a stronger statement. Namely, we also prove

\begin{thmx}\label{thm:ergodicity_refined}
For $n\geq 5$, there exists a proper subgroup $\mathcal H$ of $\PMCG(\Sigma_{0,n})$ whose action on $\Rep^{\textnormal{\tiny{DT}}}_\alpha(\Sigma_{0,n},\psl)$ is ergodic with respect to the Goldman symplectic measure. Moreover, $\mathcal H$ can be chosen to be finitely generated by $2(n-3)$ Dehn twists.
\end{thmx}

A complete parametrization of $\Rep^{\textnormal{\tiny{DT}}}_\alpha(\Sigma_{0,n},\psl)$ will be detailed in a companion paper. It consists of action-angle coordinates which are Darboux coordinates for the Goldman symplectic form on an open and dense subspace of $\Rep^{\textnormal{\tiny{DT}}}_\alpha(\Sigma_{0,n},\psl)$. Among the $2(n-3)$ Hamiltonian flows relevant for the proof of Theorem \ref{thm:ergodicity_refined}, half can be chosen to be the flows of the action coordinates of $\Rep^{\textnormal{\tiny{DT}}}_\alpha(\Sigma_{0,n},\psl)$.

\subsection{A brief note on character varieties}\label{sec:character_varieties}
Let $\surface$ denote a connected oriented topological surface of genus $g\geq 0$ and with $n\geq 0$ labelled punctures. Let $G$ denote a connected Lie group. The \textit{character variety} associated to the pair $(\surface,G)$ is the Hausdorffization of the topological quotient of the space of group homomorphisms $\pi_1(\surface)\to G$ by the conjugacy action of $G$:
\[
\Rep(\surface,G):=\Hom(\pi_1(\surface),G)/G.
\]
Group homomorphisms $\pi_1(\surface)\to G$ are also called \textit{representations} of (the fundamental group of) $\surface$ into $G$.

If $n>0$, the character variety is typically partitioned into \textit{relative character varieties}. Consider all the $n$-tuples $C=(C_1,\ldots,C_n)$ of conjugacy classes in $G$. We denote by $\Rep_C(\surface,G)$ the subset of $\Rep(\surface,G)$ consisting of conjugacy classes of representations that map a designated positively oriented loop enclosing the $i$th puncture of $\surface$ to an element inside $C_i$, for every $i=1,\ldots,n$.

To uniformize notation, we denote the character variety $\Rep(\Sigma_{g,0},G)$ of the closed surface $\Sigma_{g,0}$ as a relative character variety by $\Rep_C(\Sigma_{g,0},G)$ where $C=\emptyset$. For a detailed introduction to (relative) character varieties and their topology we refer the reader to
\cite{Mon17}.

The \textit{mapping class group} of $\surface$ is the group of isotopy classes of orientation-preserving homeomorphisms $\surface\to \surface$ that fix each puncture (if any) individually. It is denoted\footnote{In the terminology of 
\cite{FaMa12}, it is called the \textit{pure mapping class group} of $\surface$ and is denoted by PMod($\surface$). It contrasts with the mapping class group of $\surface$ where homeomorphisms are allowed to permute punctures.}  by $\PMCG(\surface)$. Standard facts about $\PMCG(\surface)$ include the following, see e.g.\ 
\cite[\S 4,\S 8]{FaMa12} for details.
\begin{itemize}
\item The mapping class group is finitely presented. Generators can be chosen to be Dehn twists along simple closed curves on $\surface$. 
\item The Dehn-Nielsen Theorem identifies $\PMCG(\surface)$ with a subgroup of the group of outer automorphisms of $\pi_1(\surface)$. 
\end{itemize}
The latter stresses a natural $\PMCG(\surface)$-action on $\Rep(\surface,G)$ by precomposition. Any homeomorphism considered in $\PMCG(\Sigma_{g,n})$ fixes each puncture individually, by definition. Thus, the $\PMCG(\Sigma_{g,n})$-action preserves every relative character variety $\Rep_C(\Sigma_{g,n},G)\subset \Rep(\Sigma_{g,n},G)$.

If the Lie algebra of $G$ admits a non-degenerate symmetric $\Ad$-invariant bilinear form $B$, then the smooth locus of $\Rep_C(\surface,G)$ enjoys a natural symplectic structure 
\cite{Gol86}, \cite{GHJW97}. It is denoted by $\omega_{\mathcal G}$ and referred to as the \textit{Goldman symplectic form}. Despite the notation, the symplectic form $\omega_{\mathcal G}$ depends on the bilinear form $B$. It easily follows from the definition of $\omega_{\mathcal G}$ that $\PMCG(\surface)$ acts by symplectomorphisms on $(\Rep_C(\surface,G),\omega_{\mathcal G})$, see e.g.\ 
\cite{Gol06}. So, if $\nu_{\mathcal G}$ denotes the symplectic measure on $\Rep_C(\surface,G)$ associated to $\omega_{\mathcal G}$, then the $\PMCG(\surface)$-action preserves $\nu_{\mathcal G}$:
\begin{equation}\label{eq:mcg_action}
\PMCG(\surface)\acts (\Rep_C(\surface,G),\nu_{\mathcal G}).
\end{equation}
As a matter of fact, if $G$ is compact, then $\Rep_C(\surface,G)$ has finite symplectic volume
\cite[Thm. 7.2]{Hu95}.

\subsection{Historical remarks}\label{sec:history} 

Theorem \ref{thm:ergodicity} is the contribution of the author to a series of results about the mapping class group dynamics on character varieties. We briefly provide the reader with an overview of this field which has been studied extensively in the past decades. The list below is certainly non-exhaustive and reflects the taste of the author. We write $\Sigma_{g,n}$ for a connected and oriented surface of genus $g$ with $n$ punctures.

Goldman proved in 
\cite{Gol97} that the mapping class group action is ergodic whenever $\surface$ has negative Euler characteristic and $G$ is a Lie group whose simple factors are isomorphic to $\text{SU}(2)$.  In 
\cite{GoXi09} Goldman-Xia provided a new proof of the ergodicity for $\text{SU}(2)$-character varieties relying on the symplectic geometry of the character variety. Goldman conjectured in 
\cite[Conj. 1.3]{Gol97} that the mapping class group action is ergodic for any compact Lie group. The conjecture was proven by Pickrell-Xia in 
\cite{PiXi02, PiXi03} for all $\Sigma_{g,n}$ with negative Euler characteristic except $\Sigma_{1,1}$. Goldman-Lawton-Xia established ergodicity for $\Sigma_{1,1}$ and $G=\text{SU}(3)$ in 
\cite{GLX21}.

If $G$ is not compact, the dynamics of the mapping class group action exhibit a different behaviour. It is, for instance, long known that the mapping class group acts properly and discontinuously on Teichm\"uller space which can be realized as a connected component of $\Rep(\Sigma_{g,0},\psl)$. More generally, the action is proper on maximal and Hitchin representations \cite{Wie06}, \cite{Lab08}. Ergodic actions contrast with proper actions by producing chaotic dynamics. Goldman promotes the following dichotomy in 
\cite{Gol06}. Assume that $G$ is noncompact and semisimple. The action is expected to be "nice" on connected components of the character variety that have a "strong" geometrical meaning (such as Teichm\"uller space). On the other hand, it is expected to give rise to more "complicated" dynamics on the remaining components. He conjectured, for instance, that the action is ergodic on the non-Teichm\"uller components of $\Rep(\Sigma_{g,0},\psl)$
\cite[Conj.\ 3.1]{Gol06}. March\'e-Wolff proved in 
\cite{MaWo16, MaWo15} that the conjecture holds for $\Sigma_{2,0}$ on the connected components of Euler class $\pm 1$ and disproved the conjecture for the component of Euler class zero. They also introduce the subspace $\mathcal{NH}^k_g$ of $\Rep(\Sigma_{g,0},\psl)$ that consists of representations with Euler class $k$ which map a simple closed curve to a non-hyperbolic element of $\psl$ and prove that the action is ergodic on $\mathcal{NH}^k_g$ for $(g,k)\neq (2,0)$, see \cite[Theo.\ 1.6]{MaWo16}. This shows that Goldman's conjecture is equivalent to $\mathcal{NH}^k_g$ having full measure in the corresponding connected component.

The counterpart of Goldman's conjecture for non-closed surfaces was formulated recently by Yang. He investigated in 
\cite{Yang16} the mapping class group action on $\Rep_C(\Sigma_{g,n},\psl)$, where $C$ is any collection of parabolic conjugacy classes. In the case of a 4-punctured sphere, he proved that the action is ergodic on every connected component of non-extremal Euler class, generalizing a result known to Maloni-Palesi-Tan for the components of Euler class $\pm 1$ 
\cite{MPT15}. He further conjectured that the analogue statement holds for every punctured surfaces 
\cite[Conj. 1.4]{Yang16}.

Several authors have also considered the action of remarkable subgroups of $\PMCG(\surface)$ on character varieties. For instance, the \textit{Johnson group} is the subgroup of $\PMCG(\surface)$ generated by Dehn twists along simple closed curves which are null-homologous in $H_1(\surface;\Z)$. Goldman-Xia proved in 
\cite{GoXi11} that the action of the Johnson group on $\Rep_C(\Sigma_{1,2},\text{SU}(2))$ is ergodic for generic $C$. This result was extended to all closed surfaces $\Sigma_{g,0}$ with $g\geq 2$ by Funar-March\'e in 
\cite{FuMa13}. Another remarkable subgroup of $\PMCG(\surface)$ is the \textit{Torelli group}. If $n\leq 1$, then the Torelli group is the subgroup of $\PMCG(\surface)$ acting trivially on $H_1(\surface;\Z)$. The Johnson group is a subgroup of the Torelli group, see e.g.\ 
\cite[\S 6]{FaMa12} for more details.  Bouilly recently proved in 
\cite{Bo20} that the action of the Torelli group on each connected component of $\Rep(\Sigma_{g,0},G)$ is ergodic for any $g\geq 2$ and for any compact connected semisimple Lie group $G$.

The mapping class group action remains of interest on character varieties on which the Goldman symplectic form cannot be defined, for there are ways to define an alternative natural invariant measure, see e.g.\
\cite{Pal11} and references therein. The first kind of examples are character varieties of non-orientable surfaces.  Palesi proved in 
\cite{Pal11} that the mapping class group action is ergodic for every non-orientable surfaces with Euler characteristic at most $-2$, including punctured surfaces, and $G=\text{SU}(2)$. Maloni-Palesi-Yang studied in 
\cite{MPY18} the mapping class group action on certain representations of the 3-punctured projective plane into $\text{PGL}(2,\R)$ that map peripheral loops to parabolic isometries. They proved that the action is ergodic on most of the connected components of non-maximal Euler characteristic. They expect ergodicity to hold on the remaining components as well.

The existence of an invariant symplectic structure may also fail for certain Lie groups. An example is the group $\text{Aff}(\C)$ of affine transformations of the complex plane. Ghazouani showed in 
\cite{Gha16} that the mapping class group action on $\Rep(\Sigma_{g,0},\text{Aff}(\C))$ does not preserve any symplectic form. There exists however an invariant measure for which the mapping class group is ergodic 
\cite{Gha16}.

\subsection{Organization of the paper} Section \ref{sec:preliminaries} provides an introduction to Deroin-Tholozan representations, recalling the notion of volume of a representation and the main results of 
\cite{DeTh16}. Further in Section \ref{sec:preliminaries} we introduce the mapping class group action in details and explain how it connects to the symplectic geometry of the character variety, before ending with a short introduction to ergodic actions.

We explain in Section \ref{sec:skeleton} how the proof of Theorem \ref{thm:ergodicity} reduces to two technical lemmata that we state in Subsection \ref{sec:two_lemmas}. Their proofs are postponed to Sections \ref{sec:rectangle_trick} and \ref{sec:key_lemma}. In Remark \ref{rem:proof_second_thm}, at the end of Section \ref{sec:skeleton}, we explain how the proof of Theorem \ref{thm:ergodicity} also implies the stronger statement of Theorem \ref{thm:ergodicity_refined}.

\subsection{Acknowledgements}
I am deeply grateful to my doctoral advisers Peter Albers and Anna Wienhard whose office doors were always open for inspirational discussions. I would like to thank Nicolas Tholozan for suggesting the project. A special thank goes to my friend Andy Sanders who always showed a great moral support in the dark hours of doctoral research.

I would like to address my gratitude to Nguyen-Thi Dang and to my old friend Quentin Posva for fruitful conversations that lead to the proof of Lemma \ref{lem:rectangle_trick}. I also thank Yohann Bouilly and Julien March\'e for pointing out relevant articles and for sharing unpublished work.

This work is supported by Deutsche Forschungsgemeinschaft (DFG, German
Research Foundation) through Germany’s Excellence Strategy EXC-2181/1 - 390900948 (the Heidelberg STRUCTURES Excellence Cluster), the Transregional Collaborative Research Center CRC/TRR 191 (281071066) and the Research Training Group RTG 2229 (281869850).

\section{Preliminaries}\label{sec:preliminaries}

\subsection{Deroin-Tholozan representations}

Deroin-Tholozan representations of the fundamental group of a punctured sphere into $\psl$ were introduced in 
\cite{DeTh16}. These representations had already been studied in the case of a 4-punctured sphere by Benedetto-Goldman 
\cite{BeGo99}. The illustrations in 
\cite{BeGo99} of the various topological types of relative character varieties of representations of a 4-punctured sphere into $\psl$ are particularly enlightening. 

Recall that $\psl$ is the Lie group defined as the quotient of the group SL$(2,\R)$ of $2\times 2$ real matrices with determinant one by its center $\{\pm I\}$. Deroin-Tholozan representations form compact connected components of certain relative character varieties where simple loops are mapped to elliptic elements of $\psl$. They are holonomies of hyperbolic metrics on a punctured sphere with prescribed conical singularities. We refer the reader to 
\cite{DeTh16} for more details about the geometrization of Deroin-Tholozan representations. In this section we recall the definition and some key properties of Deroin-Tholozan representations.

Let $\Hyp$ denote the upper half-plane with its standard hyperbolic metric. Recall that $\psl$ can be identified via M\"obius transformations with the group of orientation-preserving isometries of $\Hyp$. Elliptic elements of $\psl$ are those that have a unique fixed point inside $\Hyp$. The subspace of elliptic elements in $\psl$ is diffeomorphic to $\Hyp\times (0,2\pi)$. The diffeomorphism identifies an elliptic element $A\in \psl$ with the pair consisting of its unique fixed point $\fix (A)\in \Hyp$ and the unique angle $\vartheta=\vartheta(A)\in (0,2\pi)$ such that $A$ is conjugate to
\[
\rot_{\vartheta} := 
 \pm\begin{pmatrix}
  \cos(\vartheta/2) & \sin(\vartheta/2)  \\
  -\sin(\vartheta/2) & \cos(\vartheta/2)  
 \end{pmatrix}\in \psl.
\]
The angle $\vartheta(A)$ is called the \textit{rotation angle} of $A$. The assignment $A\mapsto \vartheta(A)$ is a smooth function of the subspace of elliptic elements of $\psl$. One can extend the function $\vartheta$ to an upper semi-continuous function $\overline\vartheta\colon \psl\to [0,2\pi]$ by setting
\[
\overline\vartheta(A):=\left\{\begin{array}{ll}
\vartheta(A), &\text{ if $A$ is elliptic,}\\
0, &\text{ if $A$ is hyperbolic or positively parabolic,}\\
2\pi, &\text{ if $A$ is the identity or negatively parabolic.}
\end{array}\right.
\]
The notions of positively and negatively parabolic elements in $\psl$ are not relevant in the context of Deroin-Tholozan representations.

We abbreviate $\Sigma_n:=\Sigma_{0,n}$ the $n$-punctured sphere. The number $n$ of punctures is always assumed to be at least 3. We fix generators $\{c_1,\ldots,c_n\}$ of $\pi_1(\Sigma_n)$ such that
\[
\pi_1(\Sigma_n)=\langle c_1,\ldots,c_n\,\vert\, c_1\cdot\ldots\cdot c_n=1\rangle.
\]
Each $c_i$ is the homotopy class of a positively oriented simple closed curve that encloses the $i$th puncture of $\Sigma_n$. If $\alpha=(\alpha_1,\ldots,\alpha_n)\in (0,2\pi)^n$ denotes an $n$-tuple of angles, then we write $\Rep_\alpha(\Sigma_{n},\psl)$ for the relative character variety that consists of conjugacy classes $[\phi]$ of homomorphisms $\phi\colon \pi_1(\Sigma_n)\to\psl$ such that $\phi(c_i)$ is conjugate to $\rot_{\alpha_i}$ for every $i=1,\ldots,n$.

A powerful tool to study the topology of relative character varieties is the notion of \textit{volume of a representation} (or Toledo number) introduced by Burger-Iozzi-Wienhard in 
\cite{BIW10}. They proved  

\begin{thm}[\cite{BIW10}]\label{thm:volume}
There exists a continuous and bounded function
\[
\vol\colon \Rep(\Sigma_{n},\psl)\to \R
\]
that satisfies the following properties:
\begin{enumerate}
\item $\vol$ is locally constant on each relative character varieties,
\item $\vol$ is additive, i.e.\ if $\Sigma_n=S_1\sqcup_\gamma S_2$ is the disjoint union of two surfaces $S_1,S_2$ glued along a separating curve $\gamma$, then
\[
\vol([\phi])=\vol([\phi\restriction_{\pi_1(S_1)}])+\vol([\phi\restriction_{\pi_1(S_2)}]),
\]
\item for every $[\phi]\in \Rep(\Sigma_{n},\psl)$, there exists an integer $k([\phi])$ such that
\[
\vol([\phi])=2\pi k([\phi])-\sum_{i=1}^n\overline\vartheta(\phi(c_i)).
\]
\end{enumerate}  
\end{thm}

Deroin-Tholozan proved in 
\cite{DeTh16} that 
\begin{equation}\label{eq:inequality_euler_characteristic}
k([\phi])\leq \max\left\{n-2,\frac{1}{2\pi}\sum_{i=1}^n\overline\vartheta(\phi(c_i))\right\}.
\end{equation}

\begin{rem}
The integer $k([\phi])$ is called the \textit{relative Euler class} of $[\phi]$. It is a generalization of the notion of Euler class associated to representations of closed surfaces. In that respect, the inequality \eqref{eq:inequality_euler_characteristic} can be thought of as a generalization of the celebrated Milnor-Wood inequality.
\end{rem}

For $[\phi]\in \Rep_\alpha(\Sigma_{n},\psl)$, the last property of Theorem \ref{thm:volume} reads
\begin{equation}\label{eq:volume_euler_class}
\vol([\phi])=2\pi k([\phi])-\sum_{i=1}^n\alpha_i.
\end{equation}
Since $k([\phi])$ is an integer and $\alpha_i\in(0,2\pi)$ for all $i$, the inequality \eqref{eq:inequality_euler_characteristic} implies $k([\phi])\leq n-1$ for every $[\phi]\in\Rep_\alpha(\Sigma_{n},\psl)$. Moreover, $k([\phi])=n-1$ is possible only if 
\begin{equation}\label{eq:angle_cond}
\alpha_1+\ldots+\alpha_n>2\pi(n-1). 
\end{equation}

\begin{dfn}
An $n$-tuple $\alpha=(\alpha_1,\ldots,\alpha_n)\in (0,2\pi)^n$ that fulfils \eqref{eq:angle_cond} is said to satisfy the \textit{angles condition}. Let
\[
\lambda:= \alpha_1+\ldots+\alpha_n-2\pi(n-1).
\]
The number $\lambda\in(0,2\pi)$ is called the \textit{scaling parameter}. 
\end{dfn}

Because of \eqref{eq:volume_euler_class}, the inequality $k([\phi])\leq n-1$ is equivalent to $\vol([\phi])\leq -\lambda$. Similarly as above, $\vol([\phi])= -\lambda$ is possible only if $[\phi]\in \Rep_\alpha(\Sigma_{n},\psl)$ where $\alpha$ satisfies the angles condition.

\begin{conv}
Unless otherwise stated, any $n$-tuple $\alpha=(\alpha_1,\ldots,\alpha_n)\in (0,2\pi)^n$ below is assumed to satisfy the angles condition \eqref{eq:angle_cond}.
\end{conv}

\begin{dfn}
The \textit{character variety of Deroin-Tholozan representations} is defined to be the subspace of $\Rep_\alpha(\Sigma_{n},\psl)$ that maximizes the volume:
\[
\smreplong:=\Rep_\alpha(\Sigma_{n},\psl)\cap \vol^{-1}(\{-\lambda\}).
\]
We abbreviate $\smreplong$ by $\smrep$.
\end{dfn}

\begin{rem}\label{rem:deroin-tholozan}
Deroin and Tholozan originally called these representations \emph{supra-maximal} because their relative Euler class exceeds $-\chi(\Sigma_n)=n-2$. However, these representations do not have maximal volume and are thus not \emph{maximal} in the sense of 
Burger-Iozzi-Wienhard. They even tend to minimize the volume in absolute value. Indeed, by definition, the volume of a Deroin-Tholozan representation is $-\lambda\in(-2\pi,0)$.
The range of the volume over the whole character variety is $[-2\pi(n-2),2\pi(n-2)]$, see \cite{BIW10}. To avoid any further confusion we prefer the terminology of Deroin-Tholozan representations instead of that of supra-maximal representations.
\end{rem}

\begin{thm}[\cite{DeTh16}]\label{thm:DeTh}
The character variety of Deroin-Tholozan representations $\smrep$ is a nonempty compact connected component of the relative character variety $\Rep_\alpha(\Sigma_{n},\psl)$. It is moreover a smooth symplectic manifold of dimension $2(n-3)$.
\end{thm}

The Goldman symplectic structure $\omega_{\mathcal G}$ on $\smrep$ is the one associated to the bilinear form 
\begin{align*}
\trace\colon &\mathfrak{sl}_2\R\times \mathfrak{sl}_2\R\longrightarrow\R\nonumber \\
&(A, B)\mapsto \trace(AB).
\end{align*}
Here, we identified the Lie algebra of $\psl$ with the Lie algebra $\mathfrak{sl}_2\R$ of $2\times 2$ traceless real matrices. 

The proof of Theorem \ref{thm:DeTh} by Deroin-Tholozan is built around a beautiful application of Delzant's classification of symplectic toric manifolds. A \textit{symplectic toric manifold} is a symplectic manifold equipped with a maximal effective Hamiltonian torus action, see e.g.\ 
\cite[\S XI]{Can06} for more details. Deroin-Tholozan first constructed such an action on $(\smrep,\omega_{\mathcal G})$. They observed then that the associated moment polytope is a rescaling by $\lambda$ of the standard $(n-3)$-simplex. The standard $(n-3)$-simplex is known to be the moment polytope for the standard torus action on $\CP^{n-3}$ equipped with the Fubini-Study symplectic form $\omega_{\mathcal{FS}}$ of total volume $\pi^{n-3}/(n-3)!$.  Delzant's classification thus implies the existence of an equivariant symplectomorphism 
\begin{equation}\label{eq:symplectomorphism}
\big(\smrep,1/\lambda\cdot \omega_{\mathcal G}\big)\cong \big(\CP^{n-3},\omega_{\mathcal{FS}}\big).
\end{equation}
The isomorphism \eqref{eq:symplectomorphism} implies that $\smrep$ has finite symplectic volume. Let $\nu_{\mathcal G}$ denote the multiple of the symplectic measure on $\smrep$ associated to $\omega_{\mathcal G}$ such that $\nu_{\mathcal G}(\smrep)=1$. It follows from the definition of the volume of a representation in
\cite{BIW10} that $\smrep$ is invariant under the mapping class group action \eqref{eq:mcg_action}. Therefore, there is a well-defined measure-preserving action 
\[
\PMCG(\Sigma_n)\acts (\smrep,\nu_{\mathcal G}).
\]

We end this short introduction by stating a crucial property of Deroin-Tholozan representations: namely, Deroin-Tholozan representations are \textit{totally elliptic}. It is meant to be understood as follows.

\begin{prop}\label{prop:everything_is_elliptic}
Let $[\phi]\in \smrep$. Then the image under the representation $\phi\colon \pi_1(\Sigma_n)\to \psl$ of any non-trivial homotopy class of loops freely homotopic to a simple closed curve is elliptic.
\end{prop}
Proposition \ref{prop:everything_is_elliptic} is a slight generalization of Lemma 3.2 in \cite{DeTh16}. Even if, technically speaking, Deroin-Tholozan only consider specific curves, Proposition \ref{prop:everything_is_elliptic} is essentially proven in \cite{DeTh16}.
\begin{proof}[Proof of Proposition \ref{prop:everything_is_elliptic}]
Let $a\in \pi_1(\Sigma_n)$ be a non-trivial homotopy class of loops freely homotopic to a simple closed curve. This simple closed curve is uniquely determined up to free homotopy. In a slight abuse of notation we denote by $a$ both the homotopy class and the associated simple closed curve.

If $a$ is homotopic to a puncture, then $\phi(a)$ is elliptic by definition of the relative character variety. Otherwise, $a$ separates $\Sigma_n$ into two surfaces $S_1\sqcup_a S_2= \Sigma_n$ of negative Euler characteristic. Let $\phi_1$ and $\phi_2$ denote the restrictions of $\phi$ to $\pi_1(S_1)$ and $\pi_1(S_2)$. The curve $a$ also determines a partition of the set $\{1,\ldots,n\}$ into two subsets $J_1$ and $J_2$ of respective cardinality $m_1$ and $m_2$. Theorem \ref{thm:volume} implies
\[
\vol([\phi_i])=2\pi k([\phi_i])-\sum_{j\in J_i}\alpha_j-\overline\vartheta(\phi_i(a)),\quad i=1,2.
\]
Since $[\phi]$ is Deroin-Tholozan,
\[
\vol([\phi])=2\pi(n-1)-\sum_{i=1}^n\alpha_i.
\]
By additivity of the volume (Theorem \ref{thm:volume}), $\vol([\phi])=\vol([\phi_1])+\vol([\phi_2])$ and thus
\begin{equation}\label{eq:oula}
2\pi(n-1)+\overline\vartheta(\phi_1(a))+\overline\vartheta(\phi_2(a))=2\pi\big(k([\phi_1])+k([\phi_2])\big).
\end{equation}
Because of inequality \eqref{eq:inequality_euler_characteristic}, it holds $k([\phi_i])\leq m_i$ for $i=1,2$. So, recalling that $m_1+m_2=n$, we deduce from \eqref{eq:oula} that
\[
\overline\vartheta(\phi_1(a))+\overline\vartheta(\phi_2(a))\leq 2\pi.
\]
By construction $\phi_1(a)=\phi_2(a)^{-1}$. Thus, the sum $\overline\vartheta(\phi_1(a))+\overline\vartheta(\phi_2(a))$, being at most $2\pi$, is either $0$ or $2\pi$.

Assume first that $\overline\vartheta(\phi_1(a))+\overline\vartheta(\phi_2(a))=0$. Then both $\overline\vartheta(\phi_1(a))$ and $\overline\vartheta(\phi_2(a))$ vanish. With this extra information, our application of inequality \eqref{eq:inequality_euler_characteristic} to $[\phi_i]$ can be refined and now gives $k([\phi_i])\leq m_i-1$ for $i=1,2$. This contradicts \eqref{eq:oula}.

Assume now that $\overline\vartheta(\phi_1(a))+\overline\vartheta(\phi_2(a))=2\pi$. Then \eqref{eq:oula}, together with the inequalities $k([\phi_i])\leq m_i$ for $i=1,2$, imply that $k([\phi_i])= m_i$ for $i=1,2$. For inequality \eqref{eq:inequality_euler_characteristic} to hold for $[\phi_1]$ and $[\phi_2]$, one must necessarily have $\overline\vartheta(\phi_1(a))>0$ and $\overline\vartheta(\phi_2(a))>0$. Therefore, $\overline\vartheta(\phi_i(a))\in(0,2\pi)$ for $i=1,2$ and we conclude that $\phi(a)$ is elliptic.
\end{proof}

\begin{rem}
If $n=3$ or $n=4$, then the converse of Proposition \ref{prop:everything_is_elliptic} holds. Namely, if $[\phi]\in\Rep_\alpha(\Sigma_{n},\psl)$ is totally elliptic and $\alpha$ satisfies the angles condition \eqref{eq:angle_cond}, then $[\phi]$ is Deroin-Tholozan. This relies on the following dichotomy for the case $n=3$ 
\cite[\S 1.2]{DeTh16}. If $n=3$ and $[\phi]\in\Rep_\alpha(\Sigma_{n},\psl)$, then one of the following holds:
\begin{itemize}
\item $\alpha_1+\alpha_2+\alpha_3\in (0,2\pi]$ and $k([\phi])=1$, or
\item $\alpha_1+\alpha_2+\alpha_3\in [4\pi,6\pi)$ and $k([\phi])=2$.
\end{itemize}  
In particular, $\alpha_1+\alpha_2+\alpha_3>4\pi$ implies $k([\phi])=2$ and hence $[\phi]$ is Deroin-Tholozan. If $n=4$, consider the pants decomposition $\Sigma_4=S_1\sqcup_{b_1} S_2$, where $b_1$ is a simple closed curve in the free homotopy class of $c_2^{-1}c_1^{-1}$ (see Figure \ref{fig:symmetric_case}). Let $[\phi]\in\Rep_\alpha(\Sigma_{n},\psl)$ and denote by $[\phi_i]$ the restriction of $[\phi]$ to $\pi_1(S_i)$. Assume that $\alpha$ satisfies the angles condition. Because of the above dichotomy, it must hold $k([\phi_i])=2$ for $i=1,2$, otherwise $\alpha_1+\alpha_2+\alpha_3+\alpha_4<6\pi$, contradicting the angles conditions. Hence $[\phi]$ is Deroin-Tholozan. The same argument does not apply if $n\geq 5$ and the question whether totally elliptic representations are Deroin-Tholozan remains open.
\end{rem}

\subsection{Relation to symplectic geometry} To prove that the $\PMCG(\Sigma_n)$-action on $\smrep$ is ergodic we follow a method developed by Goldman-Xia in 
\cite{GoXi09} and used by March\'e-Wolff in 
\cite{MaWo15}. It relies essentially on the observation that a Dehn twist $\tau_a$ along a non-trivial simple closed curve $a$ on $\Sigma_n$ is closely related to some Hamiltonian flow. This crucial observation is explained in this section.

Recall that we introduced a function $\vartheta$ that maps smoothly elliptic elements in $\psl$ to their rotation angle in $(0,2\pi)$. Proposition \ref{prop:everything_is_elliptic} says that for any non-trivial homotopy class $a\in \pi_1(\Sigma_n)$ freely homotopic to a simple closed curve and any Deroin-Tholozan representation $\phi\colon \pi_1(\Sigma_n)\to \psl$, the image $\phi(a)$ is elliptic. Consider the following function
\begin{align}
\vartheta_a\colon &\smrep\longrightarrow (0,2\pi) \nonumber\\
&[\phi]\mapsto \vartheta(\phi(a)). \label{eq:hamiltonian_fct}
\end{align}
Let $\Phi_{a}^t\colon \smrep\to\smrep$ denote the Hamiltonian flow of $\vartheta_a$ at time $t\in \R$. The flow $\Phi_{a}^t$ is called the \textit{twist flow} of $(\vartheta,a)$. Twists flows were introduced by Goldman in 
\cite{Gol84}. 

Recall that $a\in \pi_1(\Sigma_n)$ determines a unique (up to free homotopy) simple closed curve which we also denote by $a$. Cutting $\Sigma_n$ along $a$ determines two surfaces $S_1\sqcup_a S_2= \Sigma_n$. The computations conducted in 
\cite[Prop. 3.3]{DeTh16} from the original definition of twist flows by Goldman show that
\begin{equation}\label{eq:twist_flow}
\Phi_{a}^{\vartheta_a([\phi])/2}([\phi])\colon c_i\mapsto \left\{\begin{array}{ll}
\phi(c_i)  & \text{if } c_i\in\pi_1(S_1),\\
\phi(a)\phi(c_i)\phi(a)^{-1} & \text{if } c_i\in\pi_1(S_2).
\end{array}\right.
\end{equation}
Goldman-Xia observed in 
\cite{GoXi09} that the representation \eqref{eq:twist_flow} corresponds precisely to the representation obtained by letting the Dehn twist $\tau_a\in \PMCG(\Sigma_n)$ along the curve $a$ act on $[\phi]$. This is the crucial observation mentioned in introduction that connects the symplectic geometry of $\smrep$ to the action of $\PMCG(\Sigma_n)$. Formally, the following holds.

\begin{prop}\label{prop:Dehn_twists_Ham_flow}
Let $a\in \pi_1(\Sigma_n)$  be a non-trivial homotopy class of loops freely homotopic to a simple closed curve on $\Sigma_n$. Then
\[
\tau_a [\phi]=\Phi_{a}^{\vartheta_a([\phi])/2}([\phi]),\quad \forall [\phi]\in\smrep.
\]
\end{prop}

Proposition \ref{prop:Dehn_twists_Ham_flow} is used as such in
\cite[Prop. 6.5]{MaWo15}. The analogue of Proposition \ref{prop:Dehn_twists_Ham_flow} for $\SU2$-character varieties can be found in 
\cite[Prop. 5.1]{GoXi09}.

\subsection{Ergodic actions}
A measure preserving action of a group $G$ on a probability measure space $(X,\mu)$ is \textit{ergodic} if for all measurable sets $U\subset X$
\[
gU=U,\quad \forall g\in G\quad \Longrightarrow \quad \mu(U)\in\{0,1\}.
\]
Ergodicity means that the dynamical system induced by the $G$-action on $X$ admits no non-trivial subsystems. Ergodic systems exhibit a certain level of chaos through their dynamics: mixing systems are ergodic and ergodic systems have almost only dense orbits (provided that the measure is Borel). The standard example of ergodic actions are irrational rotations of the circle, see e.g.\ 
\cite[Prop. 2.16]{EiWa11}.

Ergodicity can be characterized in terms of invariant functions. The regularity class of those functions can be restricted as long as it contains the indicator functions of all measurable sets. For the purpose of this note, and in view of Lemma \ref{lem:rectangle_trick}, we choose to characterize ergodicity in terms of integrable functions. 
\begin{lem}\label{lem:ergodic_1}
A measure preserving action of a group $G$ on a probability measure space $(X,\mu)$ is ergodic if and only if every $G$-invariant integrable function $f\colon X\to \R$ is constant almost everywhere.
\end{lem}
We refer the reader to\ 
\cite{EiWa11} for the proof of Lemma \ref{lem:ergodic_1} and for further consideration on ergodic actions.

Checking that a function is constant almost everywhere can be done locally. This strategy was employed by March\'e-Wolff in 
\cite{MaWo15}. The statement is the following. Assume that $X$ is a topological space and $\mu$ is a strictly positive Borel measure on $X$, i.e.\ $\mu(U)>0$ for every nonempty open set $U\subset X$.

\begin{lem}\label{lem:ergodic_2}
Let $f\colon X\to \R$ be an integrable function. Assume that there exists an open set $\E\subset X$ such that
\begin{enumerate}
\item $\E$ is connected,
\item $\mu(\E)=1$,
\item for all $x\in \E$, there exists an open set $U_x\subset \E$ containing $x$ such that $f$ is constant almost everywhere on $U_x$.
\end{enumerate}
Then $f$ is constant almost everywhere.
\end{lem}
\begin{proof}
Define the function $F\colon \E\to \R$ by
\[
F(x):=\frac{1}{\mu(U_x)}\int_{U_x}f\,d\mu.
\]
Informally, $F(x)$ is the constant value reached by $f$ almost everywhere on $U_x$. For every $y\in U_x$, the set $U_x\cap U_y$ is nonempty and thus has positive measure by assumption. Moreover
\[
\frac{1}{\mu(U_x)}\int_{U_x}f\,d\mu=\frac{1}{\mu(U_x\cap U_y)}\int_{U_x\cap U_y}f\,d\mu=\frac{1}{\mu(U_y)}\int_{U_y}f\,d\mu.
\]
So, $F(x)=F(y)$. This means that $F$ is locally constant on $\E$ (and not only almost everywhere). For $\E$ is connected, $F$ is thus constant on $\E$.

Now, because $F\restriction_{U_x}$ is constant, $f$ and $F$ coincide almost everywhere on $U_x$ for every $x\in\E$. Hence $f=F$ almost everywhere on $\E$. Since $F$ is constant on $\E$ and $\mu(\E)=1$, we conclude that $f$ is constant almost everywhere.
\end{proof}

\section{The skeleton of the proof}\label{sec:skeleton}

According to Lemma \ref{lem:ergodic_1}, it is sufficient to show that every $\PMCG(\Sigma_n)$-invariant integrable function $f\colon \smrep\to \R$ is constant almost everywhere in order to prove Theorem \ref{thm:ergodicity}. The tool for this is Lemma \ref{lem:ergodic_2}. We apply the latter by constructing an open set $\E$ that satisfies the required hypotheses for any $\PMCG(\Sigma_n)$-invariant integrable function $f\colon \smrep\to \R$. 

In this section, we first define the open set $\E$. We then state two technical lemmata, namely Lemma \ref{lem:key_lemma} and Lemma \ref{lem:rectangle_trick}. Their proofs are postponed to Sections \ref{sec:rectangle_trick} and \ref{sec:key_lemma}. In a third and last part, we prove that $\E$ satisfies all three conditions of Lemma \ref{lem:ergodic_2}, assuming that the two lemmata mentioned above hold. 

\subsection{The set $\E$}\label{sec:set_E} Recall that we fixed generators $\{c_1,\ldots,c_n\}$ of $\pi_1(\Sigma_n)$ such that $c_1\cdot\ldots\cdot c_n=1$. We introduce the following $2(n-3)$ elements of $\pi_1(\Sigma_n)$: for every $i=1,\ldots,n-3$, let
\begin{align*}
b_i&:=c_{i+1}^{-1}c_i^{-1}\cdot\ldots\cdot c_1^{-1},\\
d_i&:=c_{i+2}^{-1}c_{i+1}^{-1}.
\end{align*}
The free homotopy classes of loops corresponding to $c_i,b_i,d_i$ can be represented by oriented simple closed curves, also denoted $c_i,b_i,d_i$, as illustrated on Figure \ref{fig:curves_b_i_d_i}. 

\begin{figure}[h]
\begin{center}
\resizebox{\columnwidth}{!}{
\begin{tikzpicture}[framed,font=\sffamily,decoration={
    markings,
    mark=at position 0.7 with {\arrow{>}}}]
 \draw (0,0) to (9,0);
 \draw[postaction={decorate}] (0,0) to[bend left=15] node[left]{$c_1$} (0,1);
 \draw[dashed] (0,0) to[bend right=15] (0,1);
 \draw (0,1) to[bend right=40] (1,2);
 \draw[postaction={decorate}] (1,2) to[bend right=15]  (2,2);
 \draw (1,2) to[bend left=15] node[above]{$c_2$} (2,2);
 \draw (2,2) to[bend right=40] (3,1);
 \draw (3,1) to[bend right=40] (4,2);
 \draw[postaction={decorate}] (4,2) to[bend right=15] (5,2);
 \draw (4,2) to[bend left=15] node[above]{$c_3$} (5,2);
  \draw (5,2) to[bend right=40] (6,1);
 \draw (6,1) to[bend right=40] (7,2);
 \draw[postaction={decorate}] (7,2) to[bend right=15]  (8,2);
 \draw (7,2) to[bend left=15] node[above]{$c_4$}(8,2);
 \draw (8,2) to[bend right=40] (9,1);
 \draw[dashed] (9,0) to (10,0);
 \draw[dashed] (9,1) to[bend right=40] (10,2);
 
 \draw[postaction={decorate}, red!50!yellow, thick] (3,1) to[bend right=15] node[left, thick]{$b_1$} (3,0);
 \draw[dashed, red!50!yellow] (3,0) to[bend right=15] (3,1);
 \draw[postaction={decorate}, red!50!yellow, thick] (6,1) to[bend right=15] node[left, thick]{$b_2$} (6,0);
 \draw[dashed, red!50!yellow] (6,0) to[bend right=15] (6,1);
 \draw[postaction={decorate}, red!50!yellow, thick] (9,1) to[bend right=15] node[left, thick]{$b_3$} (9,0);
 \draw[dashed, red!50!yellow] (9,0) to[bend right=15] (9,1);
 
 \draw[postaction={decorate}, red!40!blue, thick] (5.15,1.5) to[bend left=35] (0.85,1.5) ;
 \draw[red!40!blue] (1.4,.9) node{$d_1$};
 \draw[dashed, red!40!blue] (0.85,1.5) to[bend right=28] (5.15,1.5);
 
 \draw[postaction={decorate}, red!40!blue, thick] (8.15,1.5) to[bend left=35] (3.85,1.5);
 \draw[red!40!blue] (4.7,.8) node{$d_2$};
 \draw[dashed, red!40!blue] (3.85,1.5) to[bend right=28] (8.15,1.5);
 
 \draw[postaction={decorate}, red!40!blue, thick] (10,1.1) to[bend left=35] (6.85,1.5);
 \draw[red!40!blue] (7.5,.8) node{$d_3$};
 \draw[dashed, red!40!blue] (6.85,1.5) to[bend right=28] (10,1.2);
 
 \draw (0,2) node{\Huge{$\Sigma_n$}};
\end{tikzpicture}
}
\caption{The simple closed curves $b_1,\ldots,b_{n-3}$ and $d_1,\ldots,d_{n-3}$, and the peripheral curves $c_1,\ldots,c_n$. This illustration is modelled on \cite[Fig. 2]{DeTh16}.}
\label{fig:curves_b_i_d_i}
\end{center}
\end{figure}
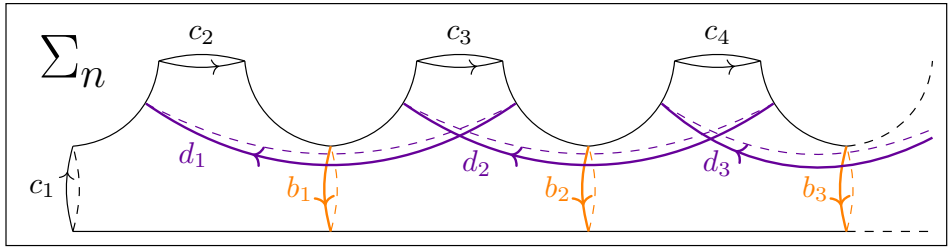

Deroin-Tholozan proved in \cite[Prop. 3.3]{DeTh16} that the Hamiltonian flows of $\vartheta_{b_1},\ldots,\vartheta_{b_{n-3}}$ are $\pi$-periodic and define a symplectic toric manifold structure on $(\smrep,\omega_{\mathcal G})$. The associated moment map $\mu:=(\vartheta_{b_1},\ldots,\vartheta_{b_{n-3}})$ maps $\smrep$ to a convex polytope $\Delta$ inside $\R^{n-3}$. We denote by $\ring\Delta$ the interior of $\Delta$. The subspace $\mu^{-1}(\ring\Delta)\subset\smrep$ is open and dense. The fibres of $\mu$ over $\ring\Delta$ are Lagrangian tori.

Because of the symplectic toric structure on $\smrep$, for any $i=1,\ldots,n-3$, the Hamiltonian flow $\Phi_{b_i}$ has the following orbit structure. Its orbits are either fixed points or circles of length $\pi$. Since any of the curves $d_1,\ldots,d_{n-3}$ can be mapped to $b_1$ by a cyclic permutation of the punctures, the Hamiltonian flows $\Phi_{d_1},\ldots,\Phi_{d_{n-3}}$ have the same orbit structure as $\Phi_{b_1}$.

\begin{dfn}
We call the orbit of $[\phi]\in\smrep$ under the combined Hamiltonian flows $\Phi_{b_1},\ldots,\Phi_{b_{n-3}}$ \textit{regular} if it is homeomorphic to an $(n-3)$-torus, equivalently if $\mu([\phi])\in\ring\Delta$. It is called \textit{irrational} if it is regular and $\vartheta_{b_i}([\phi])\in \R\setminus \pi\Q$ for every $i=1,\ldots,n-3$, equivalently $\mu([\phi])\in \ring\Delta\cap (\R\setminus \pi\Q)^{n-3}$.
\end{dfn}

As for any symplectic manifold, there is a Poisson bracket on $\smrep$ associated to $\omega_{\mathcal G}$:
\[
\big\{\cdot,\cdot\big\}\colon C^\infty\left(\smrep\right)\times C^\infty\left(\smrep\right)\to C^\infty\left(\smrep\right).
\]
It is defined as follows: for two smooth functions $\zeta_1,\zeta_2\colon \smrep\to\R$ with Hamiltonian vector fields $X_{\zeta_1},X_{\zeta_2}$, let 
\begin{equation}\label{eq:poisson_bracket}
\big\{\zeta_1,\zeta_2\big\}:=\omega_{\mathcal G}(X_{\zeta_1},X_{\zeta_2})=d\zeta_2(X_{\zeta_1}).
\end{equation}
We denote  by $(\tau_{b_i})^{m_i}d_i$ the simple closed curve obtained from $d_i$ by applying $m_i$ iterations of the Dehn twist $\tau_{b_i}$. Inspired by the work of March\'e-Wolff \cite{MaWo16}, we introduce 
\[
\E:=\left\{[\phi]\in\mu^{-1}(\ring\Delta) : \forall i=1,\ldots,n-3,\,\exists m_i\in \Z,\, \big\{\vartheta_{b_i},\vartheta_{(\tau_{b_i})^{m_i}d_i}\big\}\big([\phi]\big)\neq 0\right\}.
\]
Note that $\E\subset \smrep$ is open and measurable since
\[
\E=\mu^{-1}(\ring\Delta)\cap\bigcap_{i=1}^{n-3}\bigcup_{m_i\in \Z} \big\{\vartheta_{b_i},\vartheta_{(\tau_{b_i})^{m_i}d_i}\big\}^{-1}\big(\R\setminus \{0\}\big).
\]
We claim that $\E$ satisfies all the hypotheses of Lemma \ref{lem:ergodic_2}. This follows from Lemma \ref{lem:property1-2} and Lemma \ref{lem:property3} below.

\begin{rem}\label{rem:symmetric_case}
It is worth pointing out that the definition of the set $\E$ does not depend on the function $f\colon \smrep\to \R$ that we fixed previously. One may wonder if $\E$ is actually distinct from $\mu^{-1}(\ring\Delta)$. The answer in general remains unknown to the author; however there exist special symmetric cases where the answer is yes. 

Assume for simplicity that $n=4$. Recall that for $n=4$ the character variety of Deroin-Tholozan representations is symplectomorphic to the 2-sphere. Assume further that $\alpha_1=\alpha_2=\alpha_3=\alpha_4$. In this case, the Hamiltonian flows $\Phi_{b_1}$ and $\Phi_{d_1}$ are rotations around two perpendicular axes of the 2-sphere. We can think of the two fixed points of $\Phi_{b_1}$ as the poles of the sphere and the two fixed points of $\Phi_{d_1}$ as two diametrically opposite points on the equator (see Figure \ref{fig:symmetric_case}). Denote the fixed points of $\Phi_{d_1}$ by $[\phi_1]$ and $[\phi_2]$. The equator is the $\Phi_{b_1}$-orbit characterized by $(\vartheta_{b_1})^{-1}(\pi)$. Hence it holds $\tau_{b_1}[\phi_1]=[\phi_2]$ and $\tau_{b_1}[\phi_2]=[\phi_1]$ by Proposition \ref{prop:Dehn_twists_Ham_flow}. In particular, because the Hamiltonian vector field of $\vartheta_{d_1}$ vanishes at $[\phi_1]$ and $[\phi_2]$, it holds
\[
\big\{\vartheta_{b_1},\vartheta_{d_1}\big\}\big((\tau_{b_1})^m [\phi_1]\big)=0,\quad \forall m\in\Z.
\]
Anticipating Lemma \ref{lem:poisson_bracket_dehn_twist}, this implies 
\[
\big\{\vartheta_{b_1},\vartheta_{(\tau_{b_1})^m d_1}\big\}([\phi_1])=\big\{\vartheta_{b_1},\vartheta_{(\tau_{b_1})^m d_1}\big\}([\phi_2])=0,\quad \forall m\in\Z.
\]
Therefore $[\phi_1],[\phi_2]\in \mu^{-1}(\ring\Delta)\setminus \E$.

\begin{figure}[h!]
\begin{center}
\resizebox{9cm}{3.5cm}{%
\begin{tikzpicture}[framed,font=\sffamily,decoration={
    markings,
    mark=at position 0.7 with {\arrow{>}}}]
 \draw (0,0) to (6,0);
 \draw[postaction={decorate}] (0,0) to[bend left=15] node[left]{$c_1$} (0,1);
 \draw[dashed] (0,0) to[bend right=15] (0,1);
 \draw (0,1) to[bend right=40] (1,2);
 \draw[postaction={decorate}] (1,2) to[bend right=15]  (2,2);
 \draw (1,2) to[bend left=15] node[above]{$c_2$} (2,2);
 \draw (2,2) to[bend right=40] (3,1);
 \draw (3,1) to[bend right=40] (4,2);
 \draw[postaction={decorate}] (4,2) to[bend right=15] (5,2);
 \draw (4,2) to[bend left=15] node[above]{$c_3$} (5,2);
  \draw (5,2) to[bend right=40] (6,1);
% \draw (6,1) to[bend right=40] (7,2);
% \draw[postaction={decorate}] (7,2) to[bend right=15]  (8,2);
% \draw (7,2) to[bend left=15] node[above]{$c_4$}(8,2);
% \draw (8,2) to[bend right=40] (9,1);
% \draw[dashed] (9,0) to (10,0);
% \draw[dashed] (9,1) to[bend right=40] (10,2);
  \draw[postaction={decorate}] (6,1) to[bend right=15]  (6,0);
  \draw (6,0) to[bend right=15] node[right]{$c_4$} (6,1);
 
 \draw[postaction={decorate}, red!50!yellow, thick] (3,1) to[bend right=15] node[left]{$b_1$} (3,0);
 \draw[dashed, red!50!yellow] (3,0) to[bend right=15] (3,1);
% \draw[postaction={decorate}, blue] (6,1) to[bend right=15] node[left]{$b_2$} (6,0);
% \draw[dashed, blue] (6,0) to[bend right=15] (6,1);
% \draw[postaction={decorate}, blue] (9,1) to[bend right=15] node[left]{$b_3$} (9,0);
% \draw[dashed, blue] (9,0) to[bend right=15] (9,1);
 
 \draw[postaction={decorate}, red!40!blue, thick] (5.15,1.5) to[bend left=35] (0.85,1.5) ;
 \draw[red!40!blue] (1.4,.9) node{$d_1$};
 \draw[dashed, red!40!blue] (0.85,1.5) to[bend right=28] (5.15,1.5);
 
% \draw[postaction={decorate}, red] (8.15,1.5) to[bend left=35] node[below right]{$d_2$} (3.85,1.5);
% \draw[dashed, red] (3.85,1.5) to[bend right=28] (8.15,1.5);

	\draw (0,2) node{\Huge{$\Sigma_4$}};
 
  %\draw (1,.5) node[above]{$x$};
 %\draw (1,.5) node{$\bullet$};
\end{tikzpicture}
}
\resizebox{9cm}{7cm}{%
\begin{tikzpicture}[framed, decoration={
    markings,
    mark=at position 0.7 with {\arrow{>}}}]]
  \draw[fill=black!3] (0,0) circle (2cm);
  
   \draw[-latex] (3,-2.5) to (3,2.5); 
  %\draw (3,2.5) node[right]{$\vartheta_{b_1}$};
  \draw[dashed] (0,2) to (3,2);
  \draw (3,2) node[right]{\tiny{$\alpha_3+\alpha_4-2\pi$}};
  \draw[dashed] (2,0) to (3,0);
  \draw (3,0) node[right]{\tiny{$\pi$}};
  \draw[dashed] (0,-2) to (3,-2);
  \draw (3,-2) node[right]{\tiny{$4\pi-\alpha_1-\alpha_2$}};
  \draw[-latex] (2,1) to node[above]{$\vartheta_{b_1}$} (2.8,1) ;

  \draw[red!50!yellow] (0,2) to (0,-2);
  \draw[red!50!yellow, dashed] (0,2) to (0,3); 
  \draw[red!50!yellow, dashed] (0,-2) to (0,-3); 
  \fill[fill=red!50!yellow] (0,2) circle (2pt);
  \fill[fill=red!50!yellow] (0,-2) circle (2pt);
  
  \draw [-latex, thick, rotate=0, red!50!yellow] (-.7,2.4) arc [start angle=-190, end angle=160, x radius=.7cm, y radius=.2cm];
  \draw[red!50!yellow] (1.2,2.4) node{$\Phi_{b_1}$};
  
  \draw[postaction={decorate}, red!50!yellow] (-2,0) arc (180:360:2 and 0.6);
  \draw[red!50!yellow, dashed] (2,0) arc (0:180:2 and 0.6);
  
  \draw[postaction={decorate}, red!50!yellow] (-1,1.7) arc (180:360:1 and .2);
  
  \draw[red!40!blue] (-.57,-.57) to (.57,.57);
  \draw[red!40!blue, dashed] (-1,-1) to (-.57,-.57);
  \draw[red!40!blue, dashed] (.57,.57) to (1,1);
  \draw [-latex, thick, red!40!blue] (-.2,-.9) arc [start angle=330, end angle=-20, x radius=.5cm, y radius=.7cm];
  \draw[red!40!blue] (.4,-1) node{$\Phi_{d_1}$};
  
  \draw[postaction={decorate}, red!50!blue] (0,2) arc (90:-90:1.5 and 2);
  \draw[red!50!blue, dashed] (0,2) arc (90:270:1.5 and 2);

  \fill[fill=red!40!blue] (-.57,-.57) circle (2pt);
  \draw (-.7,-.57) node[above]{\small{$[\phi_1]$}};
  \fill[fill=red!40!blue] (.57,.57) circle (2pt);
  \draw (.4,.57) node[above]{\small{$[\phi_2]$}};
  
  \fill[fill=black] (0,0) circle (2pt); 
  
  \draw (-2.5,2) node{\Huge{$\CP^1$}};
  \draw (2.5,-3) node{\text{here: }$\alpha_1=\alpha_2=\alpha_3=\alpha_4$};
\end{tikzpicture}
}
\caption{On top: the 4-punctured sphere and the curves $b_1,d_1$. On the bottom: the flows of $\Phi_{b_1}$ and $\Phi_{d_1}$ seen as rotations around two perpendicular axes of the 2-sphere when $\alpha_1=\alpha_2=\alpha_3=\alpha_4$.}
\label{fig:symmetric_case}
\end{center}
\end{figure}
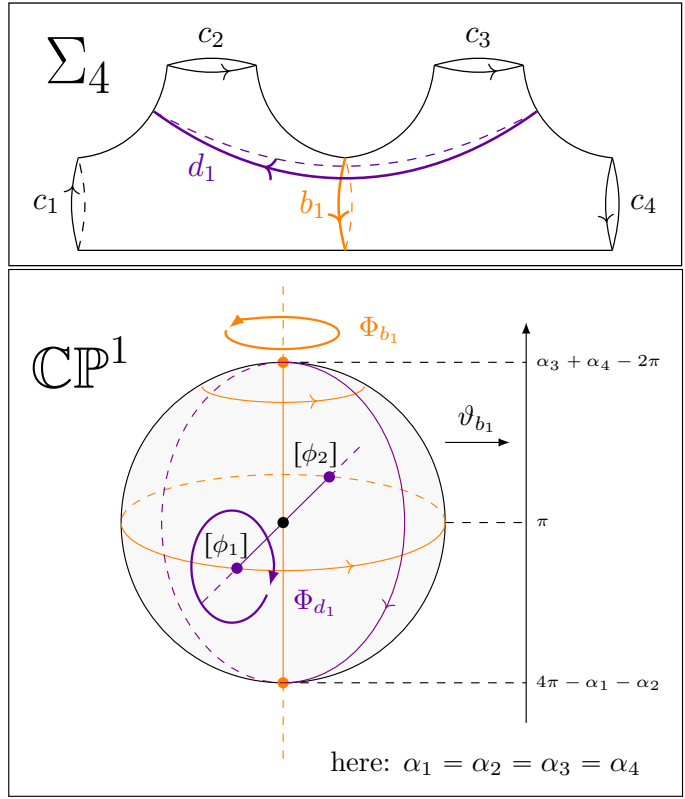
\end{rem}

\subsection{Two technical lemmata}\label{sec:two_lemmas}
The proof that $\E$ is connected and has full measure relies on the following Key Lemma and its corollary. The proof of the Key Lemma \ref{lem:key_lemma} is postponed to Section \ref{sec:key_lemma} and that of its corollary to Subsection \ref{sec:properties_1-2} below.

\begin{lem}[Key Lemma]\label{lem:key_lemma}
For every $i=1,\ldots,n-3$, every orbit of the Hamiltonian flow $\Phi_{b_i}$ contained inside $\mu^{-1}(\ring\Delta)$ contains at most two points at which $\big\{\vartheta_{b_i},\vartheta_{d_i}\big\}$ vanishes.

Moreover, if two such points exist, then they are diametrically opposite, i.e.\ they are images of each other under $\Phi_{b_i}^{t_0}$ where $t_0$ is half the minimal period of the corresponding orbit (here $t_0=\pi/2$).
\end{lem}

A particular case where two diametrically opposite points as in the conclusion of the Key Lemma \ref{lem:key_lemma} exist is described in Remark \ref{rem:symmetric_case} and illustrated on Figure \ref{fig:symmetric_case}. The Key Lemma \ref{lem:key_lemma} has the following implication on the structure of $\E$.

\begin{cor}\label{cor}
The set $\E$ contains all irrational orbits of the Hamiltonian flows $\Phi_{b_1},\ldots,\Phi_{b_{n-3}}$.
\end{cor}

The third hypothesis of Lemma \ref{lem:ergodic_2}, namely that $f$ is locally constant almost everywhere on $\E$, is a consequence of the ergodicity of irrational circle rotations and of the following result. Consider the unit hypercube $[0,1]^n\subset \R^n$. For $i=1,\ldots,n$, denote by $\pi_i\colon [0,1]^n\to [0,1]^{n-1}$ the projection map defined by forgetting the $i$th component.

\begin{lem}[Rectangle trick]\label{lem:rectangle_trick}
Let $\varphi\in L^1([0,1]^n)$. Assume that there exist full-measure sets $E_1,\ldots, E_n\subset [0,1]^{n-1}$ such that for all $i=1,\ldots,n$ and for all $x\in E_i$, $\varphi\restriction_{\pi_i^{-1}(x)}$ is constant almost everywhere. Then $\varphi$ is constant almost everywhere.
\end{lem}

The case $n=2$ of Lemma \ref{lem:rectangle_trick} reads as follows: any integrable function which is constant almost everywhere on almost every vertical and horizontal line in a rectangle is constant almost everywhere on the rectangle. Lemma \ref{lem:rectangle_trick} is certainly known to experts. However, there is a lack of concrete references in the existing literature and therefore we provide a proof of Lemma \ref{lem:rectangle_trick} in Section \ref{sec:rectangle_trick}. We now prove that $\E$ satisfies the three hypotheses of Lemma \ref{lem:ergodic_2}.

\subsection{First and second hypotheses}\label{sec:properties_1-2}

We start with a useful formula.
\begin{lem}\label{lem:poisson_bracket_dehn_twist}
Let $a$, $b$ be two simple closed curves on $\Sigma_n$. Then, for any integer $m$, it holds
\[
\big\{\vartheta_a,\vartheta_{(\tau_a)^m b}\big\}([\phi])=\big\{\vartheta_a,\vartheta_{b}\big\}\big((\tau_a)^m [\phi]\big), \quad \forall [\phi]\in\smrep.
\]
\end{lem}
\begin{proof}
Let $[\phi]\in\smrep$. It suffices to check that 
\[
\big\{\vartheta_a,\vartheta_{\tau_a b}\big\}([\phi])=\big\{\vartheta_a,\vartheta_{b}\big\}\big(\tau_a [\phi]\big).
\]
The general formula follows by induction. We compute
\begin{align*}
\vartheta_{\tau_a b}([\phi])&=\vartheta\big(\phi(\tau_a b)\big)\\
&=\vartheta\big((\tau_a \phi)(b)\big)\\
&=\vartheta_b(\tau_a[\phi]).
\end{align*}
The first and third equalities are an application of the definition of the functions $\vartheta_{\tau_a b}$ and $\vartheta_b$ (see \eqref{eq:hamiltonian_fct}). For the second equality, recall that $\PMCG(\Sigma_n)$ acts on $\smrep$ by precomposition. Using Proposition \ref{prop:Dehn_twists_Ham_flow}, we conclude that 
\begin{equation}\label{eq:02}
\vartheta_{\tau_a b}([\phi])=\vartheta_b\circ \Phi_{a}^{\vartheta_a([\phi])/2}([\phi]).
\end{equation}
Let $X_{a}$ denote the Hamiltonian vector field of $\vartheta_a$. For every time $t$ it holds
\[
X_{a}\big(\Phi^t_{a}([\phi])\big)=\big(d\Phi^t_{a}\big)_{[\phi]}\big(X_{a}([\phi])\big).
\]
In particular for $t=\vartheta_a([\phi])/2$ we get 
\begin{equation}\label{eq:01}
X_{a}(\tau_a [\phi])=X_{a}\left(\Phi^{\vartheta_a([\phi])/2}_{a}([\phi])\right)=\left(d\Phi^{\vartheta_a([\phi])/2}_{a}\right)_{[\phi]}\big(X_{a}([\phi])\big).
\end{equation}
Thus
\begin{align*}
\big\{\vartheta_a,\vartheta_{b}\big\}(\tau_a [\phi])&\stackrel{\eqref{eq:poisson_bracket}}{=}(d\vartheta_b)_{\tau_a [\phi]}\big(X_{a}(\tau_a [\phi])\big)\\
&\stackrel{\eqref{eq:01}}{=}(d\vartheta_b)_{\tau_a [\phi]}\circ \left(d\Phi^{\vartheta_a([\phi])/2}_{a}\right)_{[\phi]}\big(X_{a}([\phi])\big)\\
&\stackrel{\phantom{\eqref{eq:02}}}{=}d\left(\vartheta_b\circ \Phi^{\vartheta_a([\phi])/2}_{a}\right)_{[\phi]}\big(X_{a}([\phi])\big)\\
&\stackrel{\eqref{eq:02}}{=}(d\vartheta_{\tau_a b})_{[\phi]}(X_{a}([\phi]))\\
&\stackrel{\eqref{eq:poisson_bracket}}{=}\big\{\vartheta_a,\vartheta_{\tau_a b}\big\}([\phi]).
\end{align*}
The middle equality is an application of the chain rule. This concludes the proof of the lemma.
\end{proof}

We now proceed with the proof of Corollary \ref{cor} assuming that the Key Lemma \ref{lem:key_lemma} holds. 

\begin{proof}[Proof of Corollary \ref{cor}]
Let $[\phi]\in\smrep$ be a point on some irrational orbit of the Hamiltonian flows $\Phi_{b_1},\ldots,\Phi_{b_{n-3}}$. We want to prove that $[\phi]\in \E$. 

Assume \textit{ab absurdo} that $[\phi]\notin \E$, i.e.\ there exists $i\in\{1,\ldots,n-3\}$ such that
\[
\big\{\vartheta_{b_i},\vartheta_{(\tau_{b_i})^m d_i}\big\}([\phi])=0,\quad \forall m\in\Z.
\]
Proposition \ref{prop:Dehn_twists_Ham_flow} implies that $\vartheta_{b_i}(\tau_{b_i}[\phi])=\vartheta_{b_i}([\phi])$ and hence
\[
(\tau_{b_i})^m[\phi]=\Phi_{{b_i}}^{m\vartheta_{b_i}([\phi])/2}([\phi]),\quad \forall m\in\Z.
\]
So, by Lemma \ref{lem:poisson_bracket_dehn_twist} we obtain
\[
\big\{\vartheta_{b_i},\vartheta_{ d_i}\big\}\left(\Phi_{{b_i}}^{m\vartheta_{b_i}([\phi])/2}([\phi])\right)=0,\quad \forall m\in\Z.
\]
Since by assumption $\vartheta_{b_i}([\phi])\in \R\setminus \pi\Q$, all the points $\Phi_{{b_i}}^{m\vartheta_{b_i}([\phi])/2}([\phi])$ for $m\in \Z$ form a dense subset of the $\Phi_{{b_i}}$-orbit of $[\phi]$. Hence, by continuity, the function $\{\vartheta_{b_i},\vartheta_{ d_i}\}$ vanishes on the whole $\Phi_{{b_i}}$-orbit of $[\phi]$. This is a contradiction to the Key Lemma \ref{lem:key_lemma}. So, we conclude as expected that $[\phi]\in \E$.
\end{proof}

\begin{lem}\label{lem:property1-2}
The set $\E$ is connected and satisfies $\nu_{\mathcal G}(\E)=1$.
\end{lem}
\begin{proof}
The toric manifold structure on $\smrep$ implies that $\mu^{-1}(\ring\Delta)\subset \smrep$ is symplectomorphic to the product of $\ring\Delta$ with the standard $(n-3)$-torus. So, Corollary \ref{cor} immediately implies that $\nu_{\mathcal G}(\E)=1$ because $\ring\Delta\cap (\R\setminus \pi\Q)^{n-3}$ has full measure in $\ring\Delta$ and $\nu_{\mathcal G}(\mu^{-1}(\ring\Delta))=1$.

We now prove that $\E$ is connected. The proof essentially uses that $\ring\Delta$ is connected. Assume that $\E=A\cup B$ where $A,B\subset \E$ are open and disjoint. We prove that under these assumptions either $A$ or $B$ is empty.

By construction $\mu(\E)\subset \ring\Delta$. Corollary \ref{cor} says that any irrational orbit is entirely contained in $\E$. Moreover, the Key Lemma \ref{lem:key_lemma} also implies that any orbit $\mu^{-1}(x)$ for $x\in \ring\Delta\setminus (\R\setminus \pi\Q)^{n-3}$ must intersect $\E$. So,
\[
\mu(A)\cup\mu(B)=\ring\Delta.
\]
Tori being connected, each irrational orbit is contained either in $A$ or in $B$. Because $\ring\Delta\setminus (\R\setminus \pi\Q)^{n-3}$ is dense in $\ring\Delta$ and both $A$ and $B$ are open, $\mu^{-1}(x)\cap \E$ must be contained either in $A$ or in $B$ for every $x\in \ring\Delta\setminus (\R\setminus \pi\Q)^{n-3}$. Hence, 
\[
\mu(A)\cap\mu(B)=\emptyset.
\]
Recall that moment maps are open maps. So, both $\mu(A)$ and $\mu(B)$ are open subsets of $\ring\Delta$. Because $\ring\Delta$ is connected, it follows that either $\mu(A)$ or $\mu(B)$ is empty, and consequently that either $A$ or $B$ is empty. This concludes the proof of the lemma.
\end{proof}

\subsection{The third hypothesis}\label{sec:properties_3} We use the Rectangle Trick (Lemma \ref{lem:rectangle_trick}) to prove
\begin{lem}\label{lem:property3}
For every $[\phi]\in\E$, there exists an open neighbourhood $U_{[\phi]}\subset\E$ of $[\phi]$ such that $f$ is constant almost everywhere on $U_{[\phi]}$.
\end{lem} 
\begin{proof}
Let $[\phi]\in\E$. By definition of $\E$ there exists for every $i=1,\ldots,n-3$ an integer $m_i$ such that
\[
\big\{\vartheta_{b_i},\vartheta_{(\tau_{b_i})^{m_i}d_i}\big\}([\phi])\neq 0.
\]
This means that the tangent spaces to $\smrep$ in a neighbourhood of $[\phi]$ are generated by the $2(n-3)$ Hamiltonian vector fields
\[
X_{{b_i}},X_{{(\tau_{b_i})^{m_i}d_i}},\quad i=1,\ldots,n-3.
\]
Therefore, $[\phi]$ admits a rectangular neighbourhood $\mathcal R$ such that $\mathcal R$ is isometric to $[0,1]^{2(n-3)}$ and $\mathcal R$ is fibred perpendicularly to its faces by the flow lines of $\Phi_{{b_i}}$ and $\Phi_{{(\tau_{b_i})^{m_i}d_i}}$. Since $\E$ is open, we can assume $\mathcal R\subset \E$.

On almost all circle orbits of the Hamiltonian flows $\Phi_{{b_i}}$ and $\Phi_{{(\tau_{b_i})^{m_i}d_i}}$ crossing $\mathcal R$, the corresponding $2(n-3)$ Dehn twists
\[
\tau_{b_i},\tau_{(\tau_{b_i})^{m_i}d_i},\quad i=1,\ldots,n-3.
\]
act by irrational rotation. Indeed, this follows from Proposition \ref{prop:Dehn_twists_Ham_flow} and from full-measureness of irrational numbers. Since irrational rotations are ergodic and $f$ is by assumption $\PMCG(\Sigma_n)$-invariant, it is a consequence of Lemma \ref{lem:ergodic_1} that $f$ is constant almost everywhere on almost every orbit of the flows crossing $\mathcal R$. The Rectangle Trick (Lemma \ref{lem:rectangle_trick}) implies that $f$ is constant almost everywhere on $R$. This concludes the proof of the lemma.
\end{proof}

\begin{rem}\label{rem:proof_second_thm}
For any $i=1,\ldots,n-3$, it holds
\[
\tau_{(\tau_{b_i})^{m_i}d_i}=(\tau_{b_i})^{m_i}\tau_{d_i}(\tau_{b_i})^{-m_i}\in\PMCG(\Sigma_n).
\]
This is a general fact about Dehn twists, see e.g.\ \cite[\S 3]{FaMa12}. Therefore, we actually proved that the action of the subgroup $\mathcal H$ of $\PMCG(\Sigma_n)$ generated by the Dehn twists $\tau_{b_1},\ldots,\tau_{b_{n-3}},\tau_{d_1},\ldots,\tau_{d_{n-3}}$ on $\smrep$ is ergodic. Now, note the following. Lemma 4.1 in \cite{GhWi17} (see also 
\cite[\S 9.3]{FaMa12}) implies that the minimum number of (Dehn twist) generators of $\PMCG(\Sigma_n)$ is $\binom{n-1}{2}-1$ for $n\geq 3$ (recall that $\PMCG(\Sigma_n)$ is trivial for $n=0,1,2$). Hence, for $n\geq 5$, $\mathcal H$ is a proper subgroup of $\PMCG(\Sigma_n)$ because $\binom{n-1}{2}-1>2(n-3)$. This proves Theorem \ref{thm:ergodicity_refined}.
\end{rem}

\section{Proof of the Rectangle Trick}\label{sec:rectangle_trick}

This section is dedicated to the proof of Lemma \ref{lem:rectangle_trick}. For clarity, we only give a proof for the case $n=2$. The proof immediately generalizes to higher dimensional rectangles by induction.

The proof uses the following density result. Let $C^\infty_0 ([0,1])$ denote the space of smooth functions of the interval with zero integral and let $L^1_0 ([0,1])$ denote the space of integrable functions of the interval with zero integral.

\begin{lem}\label{lem:density}
The space $C^\infty_0 ([0,1])$ is dense inside $L^1_0 ([0,1])$.
\end{lem}
\begin{proof}
It is a well known fact that $C^\infty ([0,1])$ is dense inside $L^1 ([0,1])$. Let $\varphi \in L^1_0([0,1])\subset L^1 ([0,1])$. We want to approximate $\varphi$ with a sequence of smooth functions with zero integral. 

Because of the density of $C^\infty ([0,1])$ in $L^1 ([0,1])$, we can approximate $\varphi$ with a sequence of smooth functions $\varphi_i\in C^\infty ([0,1])$. Consider the sequence of smooth functions
\[
\widetilde\varphi_i:=\varphi_i-\int \varphi_i.
\]
By construction $\widetilde\varphi_i\in C^\infty_0 ([0,1])$. Since $\varphi$ is assumed to be integrable, the sequence of integrals $\int \varphi_i$ converges to $\int\varphi=0$. So, the sequence $\widetilde\varphi_i \in C^\infty_0 ([0,1])$ converges to $\varphi\in L^1_0([0,1])$. 
\end{proof}

\begin{proof}[Proof of Lemma \ref{lem:rectangle_trick}]
Let $\varphi\colon [0,1]\times [0,1]\to \R$ be an integrable function. We assume that $\varphi$ is constant almost everywhere on almost every vertical and horizontal segment. In other words, we assume that there exist Lebesgue measurable sets $E_h,E_v\subset [0,1]$ such that
\begin{itemize}
\item $E_h$ and $E_v$ have measure $1$,
\item $\varphi\restriction_{\{x\}\times [0,1]}$  is constant almost everywhere for every $x\in E_h$,
\item $\varphi\restriction_{[0,1]\times \{y\}}$  is constant almost everywhere for every $y\in E_v$.
\end{itemize}
We prove that under these assumptions $\varphi$ is constant almost everywhere.

Consider the functions $c^v\colon E_h\to \R$ and $c^h\colon E_v\to \R$ defined by
\[
c^v(x):=\int_0^1 \varphi(x,y)\, dy,\quad c^h(y):=\int_0^1 \varphi(x,y)\, dx.
\]
In other words, $c^v(x)$ is the value of the constant reached almost everywhere by the function $\varphi$ on the vertical segment $\{x\}\times [0,1]$, i.e.\ $\varphi(x,y)=c^v(x)$ for every $x\in E_h$ and for almost every $y\in [0,1]$. The analogue statement holds for the function $c^h$. Fubini's Theorem implies that both functions $c^h$ and $c^v$ are measurable and of class $L^1$. It is sufficient to prove that $c^h\colon E_v\to \R$ is constant almost everywhere to deduce that $\varphi\colon [0,1]\times [0,1]\to \R$ is constant almost everywhere.

For the purpose of showing that $c^h$ is constant almost everywhere, we introduce a test function $\zeta\in C^\infty_0 ([0,1])$. Using Fubini's Theorem we compute
\begin{align*}
\int_0^1\int_0^1 \varphi(x,y)\zeta(y)\, dx\, dy&=\int_0^1\zeta(y) \int_0^1 \varphi(x,y)\, dx\, dy\\
&= \int_{E_v}\zeta(y) c^h(y)\,dy
\end{align*}
Fubini's Theorem also gives
\begin{align*}
\int_0^1\int_0^1 \varphi(x,y)\zeta(y)\, dx\, dy&=\int_{E_h}\int_{E_v} \varphi(x,y)\zeta(y)\, dy\, dx\\
&= \int_{E_v}\zeta(y) \,dy \int_{E_h}c^v(x) \,dx.
\end{align*}
The last expression vanishes because $\zeta$ was chosen to have zero integral. Hence
\begin{equation}\label{eq:vanish}
\int_{E_v} \zeta(y)c^h(y)\,dy=0
\end{equation}
for every test function $\zeta\in C^\infty_0 ([0,1])$.

By Lemma \ref{lem:density} we can approximate the function $c^h-\int c^h\in L^1_0([0,1])$ with a sequence of functions $\zeta_i\in C^\infty_0 ([0,1])$. Because of \eqref{eq:vanish} we have
\[
\int_{E_v} \zeta_i(y)\left(c^h(y)-\int c^h\right)\,dy=0
\]
for every $i$. Therefore $c^h-\int c^h$ is the zero function in $L^1([0,1])$. This means that $c^h$ is constant almost everywhere and thus that $\varphi$ is constant almost everywhere.
\end{proof}

\section{Proof of the Key Lemma}\label{sec:key_lemma}
This section is dedicated to the proof of Lemma \ref{lem:key_lemma}. The proof is technical and requires to make explicit computations of the Hamiltonian vector fields $X_{b_1},\ldots,X_{b_{n-3}}$ and of the exterior derivatives of $\vartheta_{d_1},\ldots,\vartheta_{d_{n-3}}$. To that end we start with a short recap of the local structure of relative character varieties.

\subsection{Tangent spaces to relative character varieties} A representation $\phi\colon\pi_1(\Sigma_n)\to\psl$ equips the Lie algebra $\mathfrak{sl}_2\R$ of $\psl$ with the structure of a $\pi_1(\Sigma_n)$-module via
\[
\pi_1(\Sigma_n)\stackrel{\phi}{\longrightarrow}\psl\stackrel{\Ad}{\longrightarrow}\Aut(\mathfrak{sl}_2\R).
\]
The $\pi_1(\Sigma_n)$-module obtained in this way is denoted $(\mathfrak{sl}_2\R)_\phi$. A first-order deformations argument (see e.g.\ \cite{Gol86},\cite{GHJW97}) shows that the tangent space to $\Rep_\alpha(\Sigma_{n},\psl)$ at $[\phi]$ is given by the first parabolic group cohomology of $\pi_1(\Sigma_n)$ with coefficients in $(\mathfrak{sl}_2\R)_\phi$:
\begin{equation}\label{eq:tangenbt_space_isomorphic_group_cohomology}
T_{[\phi]}\Rep_\alpha(\Sigma_{n},\psl)\cong H^1_{par}\big(\pi_1(\Sigma_n);(\mathfrak{sl}_2\R)_\phi\big).
\end{equation}
The identification \eqref{eq:tangenbt_space_isomorphic_group_cohomology} depends on the choice of a preferred representative $\phi$ of the class $[\phi]$ (namely, the one that gives $\mathfrak{sl}_2\R$ the structure of a $\pi_1(\Sigma_n)$-module).

Recall that the first parabolic group cohomology of $\pi_1(\Sigma_n)$ can be defined as the quotient
\begin{equation}\label{eq:group_cohomology}
H^1_{par}\big(\pi_1(\Sigma_n);(\mathfrak{sl}_2\R)_\phi\big)=\frac{Z^1_{par}\big(\pi_1(\Sigma_n);(\mathfrak{sl}_2\R)_\phi\big)}{B^1\big(\pi_1(\Sigma_n);(\mathfrak{sl}_2\R)_\phi\big)},
\end{equation}
where
\begin{itemize}
\item $Z^1_{par}\big(\pi_1(\Sigma_n);(\mathfrak{sl}_2\R)_\phi\big)$ is the set of maps $v\colon \pi_1(\Sigma_n)\to (\mathfrak{sl}_2\R)_\phi$ satisfying the cocycle condition
\begin{equation}\label{eq:cocycle}
v(xy)=v(x)+\Ad (\phi(x))v(y),\quad \forall x,y\in \pi_1(\Sigma_n),
\end{equation}
and the coboundary conditions 
\[
\exists \xi_i\in (\mathfrak{sl}_2\R)_\phi, \quad v(c_i)=\xi_i-\Ad(\phi(c_i)) \xi_i,\quad \forall i=1,\ldots,n,
\]

\item $B^1\big(\pi_1(\Sigma_n);(\mathfrak{sl}_2\R)_\phi\big)$ is the set of maps $v\colon \pi_1(\Sigma_n)\to (\mathfrak{sl}_2\R)_\phi$ satisfying the coboundary condition
\[
\exists \xi\in (\mathfrak{sl}_2\R)_\phi, \quad v(x)=\xi-\Ad(\phi(x)) \xi,\quad \forall x\in \pi_1(\Sigma_n).
\]
\end{itemize}

Since $\smrep$ is a full dimensional connected component of the relative character variety $\Rep_\alpha(\Sigma_{n},\psl)$, the tangent space of $\smrep$ at $[\phi]$ is also identified with the first parabolic group cohomology of $\pi_1(\Sigma_n)$:
\[
T_{[\phi]}\smrep \cong H^1_{par}\big(\pi_1(\Sigma_n);(\mathfrak{sl}_2\R)_\phi\big).
\]
We consequently denote an arbitrary element of $T_{[\phi]}\smrep$ by the equivalence class $[v]$ of a cocycle $v\in Z^1_{par}\big(\pi_1(\Sigma_n);(\mathfrak{sl}_2\R)_\phi\big)$, accordingly to the quotient \eqref{eq:group_cohomology}.

\begin{rem}
To clarify the notion of parabolic group cohomology we point out the following isomorphism. Let $\widehat \Sigma_n$ denotes the surface with boundary obtained from $\Sigma_n$ by performing a real blow-up at each puncture. The boundary $\partial \widehat \Sigma_n$ of $\widehat \Sigma_n$ consists of the disjoint union of $n$ circles. The long exact sequence in cohomology for the pair $(\widehat \Sigma_n,\partial \widehat \Sigma_n)$ contains the morphism
\[
j\colon H^1(\widehat \Sigma_n,\partial \widehat \Sigma_n;\R)\to H^1(\widehat \Sigma_n;\R).
\]
The first parabolic group cohomology $H^1_{par}\big(\pi_1(\Sigma_n);\R\big)$ (with coefficients in the trivial $\pi_1(\Sigma_n)$-module $\R$) is isomorphic to the image of the morphism $j$ inside $H^1(\widehat \Sigma_n;\R)$. The analogue statement remains true for a different coefficients module. In other words, parabolic cocycles are cocycles that are exact on the boundary (without the choice of a primitive).
\end{rem}

\subsection{Some preliminary computations} Our first computations concern the zeros of the exterior derivatives of the functions $\vartheta_{d_1},\ldots,\vartheta_{d_{n-3}}$. Similar computations were already conducted in 
\cite{DeTh16}; we include them here for the sake of completeness.

\begin{lem}\label{lem:ext_deriv}
Let $a\in \pi_1(\Sigma_n)$ be a non-trivial homotopy class of loops freely homotopic to a simple closed curve. Let $[\phi]\in\smrep$ with preferred representative $\phi$ and $[v] \in H^1_{par}\big(\pi_1(\Sigma_n);(\mathfrak{sl}_2\R)_\phi\big)$ a tangent vector at $[\phi]$. Then
\[
(d\vartheta_{a})_{[\phi]}([v])=0\quad \Longleftrightarrow \quad \trace(\phi(a)v(a))=0.
\]
\end{lem}
\begin{proof}
Consider a smooth path $[\phi_t]$ inside $\smrep$ with $[\phi_0]=[\phi]$ and whose tangent vector at $t=0$ is $[v]$. Let $\vartheta_a(t):=\vartheta_a([\phi_t])$. By definition of the exterior derivative:
\[
(d\vartheta_{a})_{[\phi]}([v])=\vartheta_a'(0).
\]

We choose smooth lifts in SL$(2,\R)$ of $\phi_t(a)\in \psl$ which we also denote by $\phi_t(a)$. Since the trace is conjugacy invariant and  $\phi_t(a)$ is conjugate to $\rot_{\vartheta_a(t)}$, by definition of the function $\vartheta_a$ (see \eqref{eq:hamiltonian_fct}), it follows that
\[
2\cos(\vartheta_a(t)/2)=\pm\trace(\phi_t(a)).
\]
Applying a derivative at $t=0$ we get
\[
-2\vartheta_a'(0)\sin(\vartheta_a(0)/2)=\pm \trace(v(a)\phi(a)).
\]
Since $\vartheta_a(0)\in(0,2\pi)$ by definition of $\vartheta_a$, it follows that $\sin(\vartheta_a(0)/2)\neq 0$ and thus 
\[
\vartheta_a'(0)=0\quad \Longleftrightarrow \quad \trace(\phi(a)v(a))=0.
\]
\end{proof}

The next computation concerns the Hamiltonian vector fields $X_{b_1},\ldots,X_{b_{n-3}}$. It is convenient to introduce the following convention. Let us first fix an index $i\in\{1,\ldots,n-3\}$ with the understanding that we are working towards the proof of Lemma \ref{lem:key_lemma}. 

\begin{conv}\label{conv}
Anytime we write $[\phi]\in\smrep$ below, we assume that $\phi$ is a representative of $[\phi]$ such that the unique fixed point of $\phi(b_i)$ in the upper half-plane is the complex unit. Such a representative always exists because $\psl$ acts transitively on the upper half-plane.
\end{conv}

For convenience, we introduce the following notation
\[
\Xi:= \begin{pmatrix}
0 & 1\\
-1& 0
\end{pmatrix}\in \mathfrak{sl}_2\R.
\]
Note that $\Xi$ also belongs to SL$(2,\R)$ and projects to $\rot_\pi$ inside $\psl$. Recall moreover that
\[
\rot_t=\pm\begin{pmatrix}
  \cos(t/2) & \sin(t/2)  \\
  -\sin(t/2) & \cos(t/2)  
 \end{pmatrix}=\pm\exp\left(t/2\,\Xi\right).
\]

\begin{lem}\label{lem:Ham_vf}
The Hamiltonian vector field $X_{b_i}$ at $[\phi]\in\smrep$ is represented by the parabolic cocycle 
\[
X_{b_i}([\phi])\colon c_j\mapsto\left\{\begin{array}{ll}
0, & j=1,\ldots,i+1,\\
\Xi-\Ad(\phi(c_j))\Xi,& j=i+2,\ldots,n.
\end{array}\right.
\]
\end{lem}
\begin{proof}
The action of twist flow $\Phi_{b_i}$ on $[\phi]$ was computed in 
\cite[Prop. 3.3]{DeTh16}:
\begin{equation}\label{eq:ham_flow}
\Phi_{b_i}^t([\phi])\colon c_j\mapsto\left\{\begin{array}{ll}
\phi(c_j), & j=1,\ldots,i+1,\\
\rot_{2t}\phi(c_j)\rot_{2t}^{-1},& j=i+2,\ldots,n.
\end{array}\right.
\end{equation}
Observe that \eqref{eq:ham_flow} is a generalization of \eqref{eq:twist_flow} which is the special case $t=\vartheta_{b_i}([\phi])/2$. The Hamiltonian flow $\Phi_{b_i}$ and the vector field $X_{b_i}$ are related by
\[
\Phi_{b_i}^t([\phi])(c_j)=\exp\big(tX_{b_i}([\phi])(c_j)\big)\phi(c_j).
\]
So, $X_{b_i}([\phi])(c_j)=0$ for $j=1,\ldots,i+1$, and for $j=i+2,\ldots,n$ we compute
\begin{align*}
X_{b_i}([\phi])(c_j)&=\deriv \Phi_{b_i}^t([\phi])(c_j)\cdot  \phi(c_j)^{-1}\\
&=\Xi-\Ad(\phi(c_j))\Xi.
\end{align*}
For the last equality we used \eqref{eq:ham_flow} and $\deriv \rot_{2t}=\Xi$.
\end{proof}

We combine Lemma \ref{lem:Ham_vf} and the cocycle formula \eqref{eq:cocycle} to evaluate the parabolic cocycle $X_{b_i}([\phi])$ at $d_i=c_{i+2}^{-1}c_{i+1}^{-1}$:
\begin{align}
X_{b_i}([\phi])(d_i)&=X_{b_i}([\phi])(c_{i+2}^{-1})+\Ad(\phi(c_{i+2}^{-1}))\underbrace{X_{b_i}([\phi])(c_{i+1}^{-1})}_{=0}\nonumber\\
&=\Xi-\Ad(\phi(c_{i+2}^{-1}))\Xi.\label{eq:Hamiltonian_vf}
\end{align}

\subsection{A reformulation of the Key Lemma}
We make use of the previous computations to reformulate what it means for the Poisson bracket of $\vartheta_{b_i}$ and $\vartheta_{d_i}$ to vanish.
\begin{lem}\label{lem:trace}
The Poisson bracket $\{\vartheta_{b_i},\vartheta_{d_i}\}$ vanishes at $[\phi]\in\smrep$ if and only if
\[
\trace\left(\Xi\cdot \phi(c_{i+2}^{-1})\phi(c_{i+1}^{-1})\right)=\trace\left(\Xi\cdot \phi(c_{i+1}^{-1})\phi(c_{i+2}^{-1})\right).
\]
\end{lem}
\begin{proof}
Combining Lemma \ref{lem:ext_deriv} and \eqref{eq:Hamiltonian_vf} it follows that $\{\vartheta_{b_i},\vartheta_{d_i}\}([\phi])=0$ if and only if
\[
\trace\big(\phi(d_i)(\Xi-\Ad(\phi(c_{i+2}^{-1}))\Xi)\big)=0.
\]
Because the trace is invariant under conjugation, and since $\Ad(\phi(c_{i+2}))\phi(d_i)=\phi(c_{i+1}^{-1})\phi(c_{i+2}^{-1})$, the latter is equivalent to
\[
\trace\left(\Xi\cdot \phi(d_i)\right) =\trace\left(\Xi\cdot \phi(c_{i+1}^{-1})\phi(c_{i+2}^{-1})\right)
\]
which proves the lemma.
\end{proof}

Consider an arbitrary point $[\phi_t]:=\Phi_{b_i}^t([\phi])$ on the $\Phi_{b_i}$-orbit of $[\phi]\in\smrep$. Thanks to \eqref{eq:ham_flow}, Lemma \ref{lem:trace} implies that $\{\vartheta_{b_i},\vartheta_{d_i}\}([\phi_t])=0$ if and only if
\begin{align}
&\trace\left(\Xi\cdot\rot_{2t}\phi(c_{i+2}^{-1})\rot_{2t}^{-1}\phi(c_{i+1}^{-1})\right)\nonumber\\
=&\trace\left(\Xi\cdot\phi(c_{i+1}^{-1})\rot_{2t}\phi(c_{i+2}^{-1})\rot_{2t}^{-1}\right).\label{eq:bof}
\end{align}
What Lemma \ref{lem:key_lemma} claims is that \eqref{eq:bof} is satisfied for at most two different values of $t\in[0,\pi)$, provided that $[\phi]$ belongs to $\mu^{-1}(\ring\Delta)$. 

We now intend to compute \eqref{eq:bof} further in terms of the representation $\phi$. Let us introduce the following notation
\[
\phi(c_{i+2}^{-1})=:\pm\begin{pmatrix}
a & b\\
c & d
\end{pmatrix}, \quad \phi(c_{i+1}^{-1})=:\pm\begin{pmatrix}
x & y\\
z & w
\end{pmatrix}.
\]

\begin{lem}\label{lem:last_stone}
The relation \eqref{eq:bof} holds if and only if
\begin{align}
&\cos(2t)\left((a-d)(y+z)-(b+c)(x-w)\right)\nonumber \\
=&\sin(2t)\left((x-w)(d-a)-(b+c)(y+z)\right).\label{eq:ouff}
\end{align}
\end{lem}

The proof of Lemma \ref{lem:last_stone} is a foolish computation and is postponed to end of this section. For now, we prove Lemma \ref{lem:key_lemma} under the assumption that Lemma \ref{lem:last_stone} holds.

\begin{proof}[Proof of Lemma \ref{lem:key_lemma}]
The function $\tan(2t)$ is two-to-one for $t\in[0,\pi)$. So, if \eqref{eq:ouff} holds for at least three different values of $t$ in $[0,\pi)$, then one must have 
\begin{equation}\label{eq:system}
\left\{\begin{array}{lll}
(a-d)(y+z)&=&(b+c)(x-w), \text{ and}\\
(x-w)(d-a)&=&(b+c)(y+z).
\end{array}\right.
\end{equation}
We claim that the system \eqref{eq:system} only has trivial solutions over the real numbers, namely
\[
\left\{\begin{array}{lll}
a=d&\text{ and } &b=-c,\text{ or}\\
x=w&\text{ and }&y=-z.
\end{array}\right.
\]
Indeed, if $a=d$, then $b=-c$, or $x=w$ and $y=-z$.  Similarly, if $x=w$, then $y=-z$, or $a=d$ and $b=-c$. The case $y=-z$ leads to the analogue conclusion. If $a\neq d$, $x\neq w$ and $y\neq -z$, then
\[
\frac{y+z}{x-w}=\frac{b+c}{a-d}=-\frac{x-w}{y+z}
\]
and so $(x-w)^2+(y+z)^2=0$. This is a contradiction.

In the first case, when $a=d$ and $b=-c$, $\phi(c_{i+2}^{-1})$ commutes with $\Xi$, and thus it commutes with $\rot_\theta$ for all $\theta$. Hence, if $i\neq n-3$, then \eqref{eq:ham_flow} implies $\Phi_{b_i}^\theta([\phi])=\Phi_{b_{i+1}}^\theta([\phi])$ for every $\theta$ , and if $i=n-3$, then \eqref{eq:ham_flow} implies $\Phi_{b_{n-3}}^\theta([\phi])=[\phi]$ for every $\theta$. Both conclusions are in contradiction with the assumption that $[\phi]\in\mu^{-1}(\ring\Delta)$.

In the second case, when $x=w$ and $y=-z$, $\phi(c_{i+1}^{-1})$ commutes with $\Xi$. An analogue argument to the previous case leads to a contradiction.

Therefore, there are at most two different $t_1,t_2\in[0,\pi)$ that satisfy \eqref{eq:ouff}. Moreover, if they exist, then $\vert t_2-t_1\vert=\pi/2$ and the corresponding points on the $\Phi_{b_i}$-orbit are diametrically opposite. This concludes the proof of Lemma \ref{lem:key_lemma}.
\end{proof}

\subsection{A last computation} It remains to prove Lemma \ref{lem:last_stone} to conclude the proof of Theorem \ref{thm:ergodicity}.

\begin{proof}[Proof of Lemma \ref{lem:last_stone}]
To simplify the notation we will abbreviate $\co =\cos (t)$ and $\s=\sin(t)$. We first compute the left-hand side of \eqref{eq:bof}, namely
\begin{equation}\label{eq:trace_lhs}
\trace\left(\begin{pmatrix}
0 & 1\\
-1& 0
\end{pmatrix}\begin{pmatrix}
\co & \s\\
-\s & \co
\end{pmatrix}
\begin{pmatrix}
a & b\\
c & d
\end{pmatrix}
\begin{pmatrix}
\co & -\s\\
\s & \co
\end{pmatrix}
\begin{pmatrix}
x & y\\
z & w
\end{pmatrix}\right).
\end{equation}
First, note that
\[
\begin{pmatrix}
\co & \s\\
-\s& \co
\end{pmatrix}
\begin{pmatrix}
a & b\\
c& d
\end{pmatrix}=\begin{pmatrix}
a\co +c\s & b\co +d\s\\
-a\s+c\co & -b\s+d\co
\end{pmatrix}
\]
and
\[
\begin{pmatrix}
\co & -\s\\
\s& \co
\end{pmatrix}
\begin{pmatrix}
x & y\\
z & w
\end{pmatrix}=\begin{pmatrix}
x\co-z\s & y\co-w\s\\
x\s+z\co & y\s+w\co
\end{pmatrix}.
\]
Hence we have
\[
\begin{pmatrix}
\co & \s\\
-\s& \co
\end{pmatrix}
\begin{pmatrix}
a & b\\
c& d
\end{pmatrix}
\begin{pmatrix}
\co & -\s\\
\s& \co
\end{pmatrix}
\begin{pmatrix}
x & y\\
z & w
\end{pmatrix}=\begin{pmatrix}
\star & l_1\\
l_2 & \star
\end{pmatrix},
\]
where
\[
\left\{\begin{array}{l}
l_1= ay\co^2-aw\co\s+cy\co\s-cw\s^2+by\co\s+bw\co^2+dy\s^2+dw\co\s, \\
l_2=-ax\co\s+az\s^2+cx\co^2-cz\co\s-bx\s^2-bz\co\s+dx\co\s+dz\co^2.
\end{array}\right.
\]
So, \eqref{eq:trace_lhs} is equal to $l_2-l_1$. We now compute the right-hand side of \eqref{eq:bof}, namely
\begin{equation}\label{eq:trace_rhs}
\trace\left(\begin{pmatrix}
0 & 1\\
-1& 0
\end{pmatrix}\begin{pmatrix}
x & y\\
z & w
\end{pmatrix}
\begin{pmatrix}
\co & \s\\
-\s& \co
\end{pmatrix}
\begin{pmatrix}
a & b\\
c& d
\end{pmatrix}
\begin{pmatrix}
\co & -\s\\
\s& \co
\end{pmatrix}\right).
\end{equation}
Because the trace is conjugacy invariant, \eqref{eq:trace_rhs} is equal to 
\[
\trace\left(\begin{pmatrix}
\co & -\s\\
\s& \co
\end{pmatrix}
\begin{pmatrix}
0 & 1\\
-1& 0
\end{pmatrix}\begin{pmatrix}
x & y\\
z & w
\end{pmatrix}
\begin{pmatrix}
\co & \s\\
-\s& \co
\end{pmatrix}
\begin{pmatrix}
a & b\\
c& d
\end{pmatrix}
\right).
\]
Since $\Xi$ and $\rot_{2t}$ commute, \eqref{eq:trace_rhs} is further equal to
\[
\trace\left(
\begin{pmatrix}
0 & 1\\
-1& 0
\end{pmatrix}\begin{pmatrix}
\co & -\s\\
\s& \co
\end{pmatrix}
\begin{pmatrix}
x & y\\
z & w
\end{pmatrix}
\begin{pmatrix}
\co & \s\\
-\s& \co
\end{pmatrix}
\begin{pmatrix}
a & b\\
c& d
\end{pmatrix}
\right).
\]
Now, we can use the previous computations to get
\[
\begin{pmatrix}
\co & -\s\\
\s& \co
\end{pmatrix}
\begin{pmatrix}
x & y\\
z & w
\end{pmatrix}
\begin{pmatrix}
\co & \s\\
-\s& \co
\end{pmatrix}
\begin{pmatrix}
a & b\\
c& d
\end{pmatrix}=\begin{pmatrix}
\star & r_1\\
r_2 & \star
\end{pmatrix},
\]
where
\[
\left\{\begin{array}{l}
r_1= bx\co^2+dx\co\s-bz\co\s-dz\s^2-by\co\s+dy\co^2+bw\s^2-dw\co\s, \\
r_2= ax\co\s+cx\s^2+az\co^2+cz\co\s-ay\s^2+cy\co\s-aw\co\s+cw\co^2.
\end{array}\right.
\]
So, \eqref{eq:trace_rhs} is equal to $r_2-r_1$.

Therefore, \eqref{eq:bof} holds if and only if $l_2-l_1=r_2-r_1$. It holds $l_2-l_1=r_2-r_1$ if and only if
\begin{align*}
&-ax\co\s+az\s^2+cx\co^2-cz\co\s-bx\s^2-bz\co\s+dx\co\s+dz\co^2\\
&-ay\co^2+aw\co\s-cy\co\s+cw\s^2-by\co\s-bw\co^2-dy\s^2-dw\co\s\\
=&ax\co\s+cx\s^2+az\co^2+cz\co\s-ay\s^2+cy\co\s-aw\co\s+cw\co^2\\
&-bx\co^2-dx\co\s+bz\co\s+dz\s^2+by\co\s-dy\co^2-bw\s^2+dw\co\s.
\end{align*}
We group all the terms containing $\co\s$ on the left-hand side and all the terms containing $\co^2$ and $\s^2$ on the other side:
\begin{align*}
&2\co\s(-ax-cz-bz+dx+aw-cy-by-dw)\\
=&(\co^2-\s^2)(-cx+az+ay+cw-bx-dz-dy+bw).
\end{align*}
We factorize and use that $\co^2-\s^2=\cos(2t)$ and $2\co\s=\sin(2t)$:
\begin{align*}
&\sin(2t)\big((x-w)(d-a)-(b+c)(y+z)\big)\\
=&\cos(2t)\big((a-d)(y+z)-(b+c)(x-w)\big).
\end{align*}
This finishes the proof of the lemma.
\end{proof}

\bibliographystyle{amsalpha}
\bibliography{literature}

\end{document}